\documentclass{article}
\usepackage[a4paper,left=4cm,right=4cm,top=3cm, bottom=3.5cm]{geometry}

\usepackage{mathtools}
\usepackage{amsthm}
\usepackage{thmtools} 

\usepackage{amssymb}
\usepackage{dsfont}

\usepackage{slashed} 
\usepackage{mathdots} 

\usepackage{tikz}
\usetikzlibrary{cd}

\usepackage{microtype} 
\usepackage[pdfusetitle, hidelinks]{hyperref} 

\theoremstyle{plain}
\newtheorem{theorem}{Theorem}[section]
\newtheorem*{theorem*}{Theorem}
\newtheorem{corollary}[theorem]{Corollary}
\newtheorem{lemma}[theorem]{Lemma}

\newtheorem{proposition}[theorem]{Proposition}

\theoremstyle{definition}
\newtheorem{definition}[theorem]{Definition}
\newtheorem{remark}[theorem]{Remark}
\newtheorem{example}[theorem]{Example}
\newtheorem*{example*}{Example}

\newcommand{\bN}{\mathbb{N}} 
\newcommand{\bZ}{\mathbb{Z}} 
\newcommand{\bR}{\mathbb{R}} 
\newcommand{\bC}{\mathbb{C}}

\newcommand{\cA}{\mathcal{A}}

\newcommand{\cE}{\mathcal{E}}

\newcommand{\cK}{\mathcal{K}}
\newcommand{\cL}{\mathcal{L}}

\newcommand{\into}{\xhookrightarrow{}} 
\newcommand{\ideal}{\vartriangleleft{}} 
\newcommand{\pmat}[1]{\begin{pmatrix}#1\end{pmatrix}}
\newcommand{\cl}[1]{\overline{#1}} 
\newcommand{\interior}[1]{{#1}^{\circ}}

\renewcommand{\d}{\partial}
\newcommand{\dbar}{\bar{\partial}}
\newcommand{\rd}{\mathrm{d}} 

\newcommand{\pt}{\mathrm{pt}} 
\newcommand{\loc}{\mathrm{loc}} 
\newcommand{\Cst}{\ensuremath{C^*}}
\newcommand{\id}{\mathds{1}} 

\newcommand{\dirlim}{\varinjlim} 

\DeclareMathOperator{\ind}{ind} 

\DeclareMathOperator{\Dom}{Dom}

\DeclareMathOperator{\Coker}{Coker} 

\DeclareMathOperator{\supp}{supp}

\DeclareMathOperator{\essspec}{Spec_{ess}} 

\DeclareMathOperator{\sgn}{sgn} 
\DeclareMathOperator{\Tr}{Tr} 

\DeclareMathOperator{\wind}{wind}

\DeclarePairedDelimiterX{\set}[1]{\lbrace}{\rbrace}{\,#1\,}
\DeclarePairedDelimiterX{\setcond}[2]{\lbrace}{\rbrace}{\,#1\,\mathclose{}\delimsize :\mathopen{}\,#2\,}

\DeclarePairedDelimiter{\braket}{\langle}{\rangle}
\DeclarePairedDelimiter{\bra}{\langle}{\rvert}
\DeclarePairedDelimiter{\ket}{\lvert}{\rangle}
\newcommand{\ketbra}[2]{\ket{#1}\bra{#2}}

\DeclarePairedDelimiter{\norm}{\lVert}{\rVert}
\DeclarePairedDelimiter{\abs}{\lvert}{\rvert}

\DeclareFontFamily{OT1}{pzc}{}
\DeclareFontShape{OT1}{pzc}{m}{it}{<-> s * [1.2] pzcmi7t}{}
\DeclareMathAlphabet{\mathpzc}{OT1}{pzc}{m}{it}

\newcommand{\pE}{\mathpzc{E}}
\newcommand{\pF}{\mathpzc{F}}
\newcommand{\pG}{\mathpzc{G}}

\DeclareMathOperator{\CompInv}{CompInv}

\numberwithin{equation}{section}

\title{The relative index theorem and a characterization of Fredholm operators}
\author{Magnus Fries \\ \textit{Centre for Mathematical Sciences, Lund University, Sweden}}
\date{\today}

\begin{document}

\maketitle

\begin{abstract}

We extend the relative index theorem on non-compact manifolds to encompass a wide variety of hypoelliptic differential operators of arbitrary order, demonstrating that the change in index when changing a differential operator locally can be calculated locally. We also show that the notion of invertibility at infinity (and coercive at infinity) is not only sufficient condition for an operator to be Fredholm but also necessary, resulting in a general geometric characterization of Fredholmness. This characterization connects to a model for unbounded \(KK\)-theory which assumes the operator to be Fredholm instead of having (locally) compact resolvent, and thus provides a convenient tool for index theory on non-compact spaces.

{
~\newline \noindent 
\textbf{MSC Classification: } 
\href{https://mathscinet.ams.org/mathscinet/msc/msc2020.html?t=19K56}{Index theory [19K56]},
\href{https://mathscinet.ams.org/mathscinet/msc/msc2020.html?t=19K35}{Kasparov theory (\(KK\)-theory) [19K35]},
\href{https://mathscinet.ams.org/mathscinet/msc/msc2020.html?t=58B34}{Noncommutative geometry (à la Connes) [58B34]}, \href{https://mathscinet.ams.org/mathscinet/msc/msc2020.html?t=47A53}{(Semi-) Fredholm operators; index theories [47A53]}.
}

\end{abstract}

\tableofcontents

\section{Introduction}

Index theory for elliptic differential operators on closed manifolds is understood using the Atiyah-Singer index theorem \cite{Atiyah_Singer_1963}. For non-compact manifolds and operators that are not elliptic, the theory is more involved. Instead of trying to compute the index of a single operator on a non-compact manifold, an approach originating from Gromov-Lawson's work on scalar curvature \cite{Gromov_Lawson_1983} is to instead look at the difference in index of two differential operators that agree outside compact sets. Then with an appropriate gluing procedure the index difference can be calculated using the Atiyah-Singer index theorem on the compact sets. Such a so-called relative index theorem first appeared in \cite{Gromov_Lawson_1983} for Dirac operators and later generalized by Bunke in \cite{Bunke_1995} in the framework of Kasparov's \(KK\)-theory \cite{Kasparov_1980}. Related work can also be found in \cite{Ballmann_Bar_2012, Nazaikinskii_2015, Bandara_2022}.

In this paper, we show that the relative index theorem in \cite{Bunke_1995} can be proven in great generality without the use of \(KK\)-theory. This demonstrates that the index of a local operator, such as a differential operator, is local in the sense that the change in index when changing the operator locally can be computed locally. We can then connect the relative index theorem to recent extensions of the Atiyah-Singer index theorem (see for example \cite{Mohsen_2022a,Ewert_2023, mohsen2022indexmaximallyhypoellipticdifferential,goffeng2024indextheoryhypoellipticoperators}) to differential operators in the Heisenberg pseudodifferential calculus on filtered manifolds (also referred to as Carnot manifolds) developed in \cite{van_Erp_Yuncken_2019} (see also \cite{Dave_Haller_2022}). Our results produce a computational formula for index differences of hypoelliptic operators on possibly non-compact manifolds that agree outside compact sets.

We also revisit the assumption of \emph{invertibility at infinity} for a differential operator \(D\) used in \cite{Bunke_1995} to produce bounded \(KK\)-cycles, meaning that there is an \(f\in C_c^\infty(M)\) such that \(f+D^\dagger D\) is invertible. A similar assumption was present in the work of Gromov-Lawson \cite[Assumption 4.16]{Gromov_Lawson_1983} and invertibility at infinity also relates to the notion of \emph{coercive at infinity} in \cite[Definition 8.2]{Ballmann_Bar_2012} and \cite{Bandara_2022} which is a norm estimate for sections supported outside a compact set. The inspiration for invertibility at infinity comes from a Dirac operator \(\slashed{D}\) on a spin manifold for which we have the Bochner-Lichnerowicz formula \(\slashed{D}^2 = \nabla^\dagger \nabla + \frac{\kappa}{4}\), where \(\nabla^\dagger \nabla\) is the connection Laplacian and \(\kappa\) is the scalar curvature \cite[Theorem 8.8]{Lawson_Michelsohn_1989}. For any \(f\in C^\infty(M)\) such that \(\frac{\kappa}{4} + f \geq \epsilon > 0\) we see that \(f+\slashed{D}^2\) is invertible, and in particular if the scalar curvature is (strictly) positive outside a compact set we can find such an \(f\) in \(C_c^\infty(M)\). Invertibility at infinity directly implies (left-)Fredholmness, and we will show that it is in fact a characterization of (left-)Fredholmness. Hence, showing that the assumption made in \cite{Gromov_Lawson_1983, Bunke_1995, Ballmann_Bar_2012} is not only sufficient but indeed necessary for Fredholmness. This also lets us prove a generalization of the classical result Persson's lemma in spectral theory describing the essential spectrum of a self-adjoint operator bounded from below.

Our characterization of Fredholmness enables us to naturally connect the bounded \(KK\)-cycles in \cite{Bunke_1995} to Wahl's \emph{truly unbounded Kasparov modules} in \cite[Definition 2.4]{Wahl_2007}. Truly unbounded Kasparov modules differ from the usual definition of unbounded Kasparov modules in that they do not require locally compact resolvent and instead assumes that the operator is Fredholm. As such, truly unbounded Kasparov modules can always be extended to a unital \Cst-algebra and creates a convenient tool for using \(KK\)-theory in index theory for non-compact spaces. For instance, in Lafforgue's work on the Baum-Conne conjecture in the case of groups acting properly on
manifolds with negative or zero sectional curvature bounded from below, the representative used for the \(\gamma\)-element \cite[Proposition 3.1.2]{Lafforgue_2002} can be seen as coming from a truly unbounded Kasparov module.

The geometric setup for the relative index theorem starts of with two manifolds \(M_{11}\) and \(M_{22}\) (possibly with boundary) with open covers \(M_{11}=V_1\cup W_1\) and \(M_{22}=V_2\cup W_2\) such that there is a diffeomorphism \(V_1\cap W_1\simeq V_2\cap W_2\). We will identify the intersections with a manifold \(U\), which we assume is a non-empty manifold without boundary embedded precompactly in \(M_{11}\) and \(M_{22}\). See \autoref{fig:manifold_parts_pic} for a schematic illustration of \(M_{11}\) and \(M_{22}\). Now, \(V_1\) can be glued together at \(U\) to either \(V_2\) or \(W_2\) and retain a manifold structure, and we will assume that this is true for \(W_2\). That is, we assume that the pushout \(M_{12}\coloneq V_1\cup_U W_2\) inherits a manifold structure. Then \(M_{21} = V_2 \cup_U W_1\) is necessarily also a manifold. 

\begin{figure}[ht]
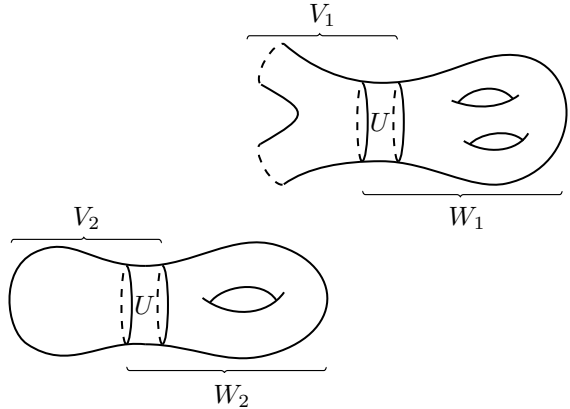

    \centering
    \setlength{\unitlength}{208.52128637bp}
    \begin{picture}(1,0.73557464)
        \setlength\tabcolsep{0pt}
        \put(0,0){\includegraphics[width=\unitlength,page=2]{manifold_parts.pdf}}
        \put(0,0){\includegraphics[width=\unitlength,page=3]{manifold_parts.pdf}}
        \put(0.64664495,0.50125503){\(U\)}
        \put(0.22479986,0.17166377){\(U\)}
        \put(0.53708349,0.69711816){\(V_1\)}
        \put(0.78716117,0.33175793){\(W_1\)}
        \put(0.11511602,0.33113108){\(V_2\)}
        \put(0.36656024,0.01118359){\(W_2\)}
    \end{picture}
    \caption{A schematic illustration of the manifolds \(M_{11}\) and \(M_{22}\) that coincide on an open set \(U\).}
    \label{fig:manifold_parts_pic}
\end{figure}

Another way to view this geometric setup is that we have a diagram of manifolds 
\begin{equation}
    \label{eq:the_diagram}
    \begin{tikzcd}
        M_{11} & V_1 \ar{r} \ar{l} & M_{12} \\
        W_1 \ar{u} \ar{d} & U \ar{r} \ar{u} \ar{l} \ar{d} & W_2 \ar{u} \ar{d} \\
        M_{21} & V_2 \ar{r} \ar{l} & M_{22}
    \end{tikzcd}
\end{equation}
where the arrows represent maps that are diffeomorphisms onto their images and that \(M_{ij}=V_i\cup_U W_j\) and \(U=V_i\cap W_j\) for \(i,j=1,2\) if we identify the image with the domain of these maps. Hence, each \(M_{ij}=V_i\cup_U W_j\) is a pushout construction that we assume retain a smooth structure. See \autoref{fig:schematic_pic} for a schematic illustration visualizing the diagram \eqref{eq:the_diagram}. A natural example for a diagram as \eqref{eq:the_diagram} is if \(\Sigma\) is a closed manifold embedded into both \(M_{11}\) and \(M_{22}\) with co-dimension one such that the complement of the image consists of two disjoint open manifolds. Then \(U\) can be taken to be a tubular neighborhood \(\Sigma \times (-1,1)\) of \(\Sigma\) and one could glue the disjoint pieces together to obtain \(M_{12}\) and \(M_{21}\).
\begin{figure}[ht]
    \centering
    \setlength{\unitlength}{253.48789485bp}
    \begin{picture}(1,0.46391964)
        \put(0,0){\includegraphics[width=\unitlength,page=2]{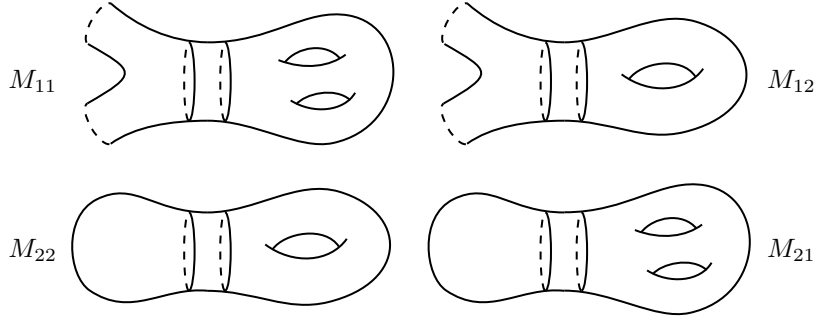}}
        \put(-0.09209692,0.33735943){\(M_{11}\)}
        \put(-0.09209692,0.08602645){\(M_{22}\)}
        \put(1.0240714,0.33735943){\(M_{12}\)}
        \put(1.0240714,0.08602645){\(M_{21}\)}
    \end{picture}
    \caption{A schematic illustration of the manifolds \(M_{11}, M_{22}, M_{12}\) and \(M_{21}\) in the geometric setup for the relative index theorem.}
    \label{fig:schematic_pic}
\end{figure}

On each manifold \(M_{ij}\) we will consider a differential operator \(D_{ij}\) in such a way that they agree where they are defined on the same open sets \(V_i\) and \(W_j\). We will allow for \(D_{ij}\) to have coefficients in a unital \Cst-algebra \(B\) and act on sections of vector \(B\)-bundles, meaning bundles of finitely generated projective \(B\)-modules. For such \(B\)-linear operators, we will use the notion of \(B\)-Fredholm meaning invertible modulo \(B\)-compact operator as an operator from its domain. A \(B\)-Fredholm operator has as an index valued in the \(K\)-theory group \(K_0(B)\) and if \(B=\bC\) then \(B\)-Fredholm coincides with the usual notion of Fredholm on a Hilbert space. See \cite{Miscenko_Fomenko_1979} for more on differential operators with coefficients in a \Cst-algebra. 

Since we allow the manifolds \(M_{ij}\) to have boundaries we will also consider boundary conditions, which will correspond to a choice of closed realizations of each \(D_{ij}\). We do not require these boundary conditions to be local, but they do need to separate each \(V_i\) and \(W_j\). More precisely, let \(E_{ij}\to M_{ij}\) and \(F_{ij}\to M_{ij}\) be vector \(B\)-bundles that satisfies a similar diagram as \eqref{eq:the_diagram} and consider differential operators 
\[
    D_{ij}\colon C_c^\infty(M_{ij}; E_{ij})\to C_c^\infty(M_{ij}; F_{ij})
\]
with coefficients in the \Cst-algebra \(B\). Let \(b_{V_i}\colon C^\infty(V_i;E)\to C^\infty(\tilde{\d}V_i;G_i)\) and \(b_{W_j}\colon C^\infty(W_j;E)\to C^\infty(\tilde{\d}W_j;H_j)\) be boundary operators for some vector \(B\)-bundles \(G_i \to \tilde{\d}V_i\) and \(H_j\to\tilde{\d}W_j\) where \(\tilde{\d}V_i = \d M_{ij}\cap V_i\) and \(\tilde{\d}W_j = \d M_{ij}\cap W_j\). We identify each \(D_{ij}\) with the closed realization as operators on the Hilbert \(B\)-modules \(L^2(M_{ij};E)\) with a core
\[
    \cE_{ij}\coloneq \setcond{f\in C_c^\infty(M_{ij};E)}{b_{V_i}(f|_{V_i}) = 0 \text{ and } b_{W_j}(f|_{W_j}) = 0}.
\]
Note in particular that \(\cE_{ij}\) is preserved by any \(a\in C^\infty(M_{ij})\) that is constant outside \(U\). We will also assume that these closed realizations are regular as unbounded operators on Hilbert \(B\)-modules. In case \(B=\bC\), we are dealing with Hilbert spaces and all closed densely defined unbounded operators are regular. In general however, this is not immediate. See for example \cite[Theorem 2.3]{Hanke_Pape_Schick_2015}.

\begin{theorem}
    \label{thm:relative_index_for_differential_operators}
    Let \(D_{11}, D_{22}, D_{12}, D_{21}\) be regular closed realizations of differential operators with coefficients in a unital \Cst-algebra \(B\) as in the previous paragraph. For \(i,j=1,2\) assume that 
    \begin{enumerate}
        \item \(D_{i1} = D_{i2}\) on \(V_i\);
        \item \(D_{1j} = D_{2j}\) on \(W_j\);
        \item \([D_{ij}, a]\) extends to a compact operator from \(\Dom D_{ij}\) for any \(a\in C_c^\infty(U)\).
    \end{enumerate}
    Then \(D_{11}\) and \(D_{22}\) are \(B\)-Fredholm if and only if \(D_{12}\) and \(D_{21}\) are, and if so
    \[
        \ind_B D_{11} + \ind_B D_{22} = \ind_B D_{12} + \ind_B D_{21}
    \]
    in \(K_0(B)\).
\end{theorem}

The proof of \autoref{thm:relative_index_for_differential_operators} can be found at the end of \autoref{sec:bunkesunitary}. Looking at the schematic illustration of the manifolds \(M_{ij}\) on which the differential operators \(D_{ij}\) act on in \autoref{fig:schematic_pic}, \autoref{thm:relative_index_for_differential_operators} says that the sums of indices over each column of manifolds are equal.

The assumptions on \(D_{ij}\) in \autoref{thm:relative_index_for_differential_operators} are almost the same conditions needed to construct truly unbounded Kasparov \((C(U^{2\pt}),B)\)-modules, where \(U^{2\pt}\) is a two-point compactification of \(U\), making it possible to extend \autoref{thm:relative_index_for_differential_operators} into \(KK\)-theory. Specifically, if each \(D_{ij}\) is self-adjoint and there is some \(m>0\) such that \([D_{ij}, a](1+D_{ij}^2)^{-\frac{1}{2}+\frac{1}{2m}}\) extends to a bounded operator in \(\cL(L^2(M_{ij};E))\) for all \(a\in C_c^\infty(U)\), then the equality in \autoref{thm:relative_index_for_differential_operators} holds in \(K_1(B)\) as well. Note that the condition that \([D_{ij}, a](1+D_{ij}^2)^{-\frac{1}{2}+\frac{1}{2m}}\) extends to a bounded operator is very natural for a pseudodifferential operator that is elliptic of order at least \(m\) in some calculi. Moreover, with the methods used in this paper one could just as well prove the relative index theorem in \(KO_{p,q}(B)\) for a real \Cst-algebra \(B\) simply by extending truly unbounded Kasparov modules to real \Cst-algebras. See \autoref{ex:splitting_z2} for an application of this.

To obtain a more computational relative index theorem, we will instead start with two differential operators that agree outside compact sets. Assuming one of these compact sets is a domain in a closed manifold that the principal symbol can be extended to, we can construct the forth manifold and compute the original index difference with the  Atiyah-Singer index theorem.

\begin{theorem}
    \label{thm:computational_relative_index_theorem}
    Let \(M_1\) and \(M_2\) be filtered manifolds (without boundary) that coincide outside compact sets in the sense that there are \(K_i\subseteq M_i\) compact such that \(M_1\setminus K_1 \cong M_2\setminus K_2\) as filtered manifolds. Also let \(E_i\to M_i\) and \(F_i\to M_i\) be vector bundles such that \(E_1|_{M_1\setminus K_1} \cong E_2|_{M_2\setminus K_2}\) and similarly for \(F_i\). Consider differential operators 
    \[D_i\colon C^\infty(M_i;E_i)\to C^\infty(M_i;F_i)\]
    that that satisfies that
    \begin{enumerate}
        \item \(D_i\) is Heisenberg elliptic of order \(m\) on an open neighborhood of \(K_i\);
        \item \(K_1\) is a domain in a closed filtered manifold \(\widetilde{M}_1\) over which \(D_1\) extends to a Heisenberg elliptic operator;
        \item \(D_1 = D_2\) on \(M_i\setminus K_i\).
    \end{enumerate}
    Identify each \(D_i\) with its closure on \(C_c^\infty(M_i;E_i)\), then \(D_1\) is Fredholm if and only if \(D_2\) is Fredholm and if so
    \begin{equation}
        \label{eq:relative_index_calculated_with_AtiyahSinger}
        \ind D_1 - \ind D_2 = \int_{S^*K_1} \alpha(\sigma_H^m(D_1)) - \int_{S^*K_2} \alpha(\sigma_H^m(D_2))
    \end{equation}
    where \(\alpha\) comes from the Atiyah-Singer index theorem for Heisenberg elliptic operators on a closed manifold and \(S^*K_i\) is the co-sphere bundle over \(K_i\).
\end{theorem} 

In the case of classically elliptic operators, in \autoref{thm:computational_relative_index_theorem} the closed manifold \(\widetilde{M}_1\) can easily be constructed by a doubling of \(K_1\) and extending the symbol of \(D_1\). Also, one could construct a theorem like \autoref{thm:computational_relative_index_theorem} for any calculus that have a local index formula on compact subspaces.

To place the relative index theorem in \(KK\)-theory we sought an unbounded model corresponding to the bounded cycles in \cite{Bunke_1995}. This leads us to examining the assumption invertibility at infinity assumption further and realize that it is in fact a characterization of Fredholmness. The author believes this characterization to be an important tool for index theory of independent interest.

We will use the notion left-\(B\)-Fredholm, meaning a \(B\)-linear operator on a Hilbert \(B\)-module that is left-invertible up to \(B\)-compact operators. In particular, if an operator \(D\) and its adjoint \(D^*\) are both left-\(B\)-Fredholm, then \(D\) is \(B\)-Fredholm. If \(B=\bC\) and we consider an operator on a Hilbert space, then the operator is left-\(B\)-Fredholm (or simply left-Fredholm) if it has finite dimensional kernel and closed range. For example, the gradient on a closed manifold is left-Fredholm. An abstract characterization of left-\(B\)-Fredholmness can be found as \autoref{thm:characterization_of_Fredholm_abstract} in the bulk of the text, which we can specialize to a geometric setting as follows.

\begin{theorem}
    \label{thm:characterization_of_Fredholm_geometric}
    Let \(D\colon L^2(M;E)\dashrightarrow L^2(M;F)\) be a regular operator acting on the Hilbert \(B\)-modules where \(E\) and \(F\) are vector \(B\)-bundles. Assume that \(\Dom D \subseteq H_{\loc}^s(M;E)\) for some \(s>0\) where \(H_{\loc}^s(M;E)\) denotes a local Sobolev space. Then the following are equivalent:
    \begin{enumerate}
        \item \(D\) is left-\(B\)-Fredholm;
        \item There is a positive \(k\in C_c^\infty(M)\) such that \(k + D^* D\) is invertible;
        \item There is a \(k\in C_c^\infty(M)\) such that
        \[
            \norm{f}\lesssim \norm{k f} + \norm{Df}
        \]
        for any \(f\) in a core of \(D\).
    \end{enumerate}
\end{theorem}

In the bulk of the text we prove \autoref{thm:characterization_of_Fredholm_geometric} using the abstract result presented in \autoref{thm:approx_unit_iff_Fred}. Note that in the case \(M\) is a closed manifold, all the statements in \autoref{thm:characterization_of_Fredholm_geometric} follow directly from the assumed inclusion \(\Dom D \subseteq H_{\loc}^s(M;E)\) since in that case Rellich theorem implies that \(D\) has compact resolvent and therefore discrete spectrum. Such an inclusion is expected for any (pseudo-) differential operator that is in some sense elliptic of positive order.

\autoref{thm:characterization_of_Fredholm_geometric} captures the notion of invertibility at infinity that was used in \cite{Bunke_1995}, but not the notion of coercive at infinity that was used in \cite{Ballmann_Bar_2012, Bandara_2022}. To obtain such results we need the additional assumption on our operator \(D\) that its commutator with smooth compactly supported functions is compact from the domain of \(D\), which is a natural assumption in most pseudodifferential calculi and also present in truly unbounded Kasparov modules. In particular, the notion of coercive at infinity can also be used for higher-order operators, extending to the first-order (elliptic) case treated in \cite{Ballmann_Bar_2012, Bandara_2022}.

\begin{theorem}
    \label{thm:characterization_of_Fredholm_geometric_extended}
    Let \(D\colon L^2(M;E)\dashrightarrow L^2(M;F)\) be as in \autoref{thm:characterization_of_Fredholm_geometric} and assume additionally that for any \(a\in C_c^\infty(M)\) we have that \(a\Dom D\subseteq \Dom D\) and \([D,a](1+D^*D)^{-\frac{1}{2}}\) is compact. Then \(D\) is left-\(B\)-elliptic if and only if there is a compact \(K\subseteq M\) such that
    \[
        \norm{f}\lesssim \norm{Df}
    \]
    for any \(f\) in a core of \(D\) with \(\supp f\subseteq M\setminus K\). Furthermore, let
    \[
        \Sigma \coloneq \sup_{\substack{K\subseteq M \\ \text{compact}}} \inf_{\substack{f\in \Dom D \setminus \set{0} \\ \supp f \subseteq M\setminus K}} \frac{\norm{Df}}{\norm{f}},
    \]
    then \(D\) is left-\(B\)-Fredholm if and only if \(\Sigma\neq 0\) and \(D + T\) is left-\(B\)-Fredholm for any \(T\in \cL(L^2(M;E), L^2(M;F))\) such that \(\norm{T}<\Sigma\).
\end{theorem}

An abstract version of the first statement in \autoref{thm:characterization_of_Fredholm_geometric_extended} is found as \autoref{thm:upper_bound_when_restricted_by_approx_unit} in the bulk of the text. The second statement in \autoref{thm:characterization_of_Fredholm_geometric_extended} is proven in \autoref{sec:perssonslemma} and allows for a short proof of Persson's lemma which is presented as \autoref{thm:perssonslemma}.

We can for example apply \autoref{thm:characterization_of_Fredholm_geometric} to the Dirac operator \(\slashed{D}\) on a complete spin manifold \(M\). We obtain that \((-\frac{\sqrt{\kappa_\infty}}{2}, \frac{\sqrt{\kappa_\infty}}{2})\) is guaranteed to be disjoint from the essential spectrum of \(\slashed{D}\), where \(\kappa_\infty \coloneq \sup_{\substack{K\subseteq M \\ \text{compact}}} \inf_{p\in M\setminus K} \kappa(p)\) and \(\kappa\) is the scalar curvature. 

The content of this paper is organized as follows.

In \autoref{sec:examples} we present how the relative index theorem presented in this paper relates to previous works and give examples on how one could use it to produce new index formulas.

In \autoref{sec:relative_index_theorem} we give a short introduction to Hilbert modules and Fredholm operators, and then present the proof of the relative index theorem in \autoref{thm:relative_index_for_differential_operators}.

In \autoref{sec:characterization_of_Fred} we prove our characterization of Fredholmness in an abstract setting as \autoref{thm:characterization_of_Fredholm_geometric} and \autoref{thm:characterization_of_Fredholm_geometric_extended}. We also relate this characterization to the classical results in spectral theory Weyl's criterion and Persson's lemma.

In \autoref{sec:KKtheory} we turn to \(KK\)-theory and present truly unbounded Kasparov modules. We prove that the two different ways to construct bounded cycles seen in \cite{Bunke_1995} and \cite{Wahl_2007} give the same class and briefly look at the relative index theorem in \(KK\)-theory.

\subsection*{Acknowledgment}

I wish to thank my PhD supervisor Magnus Goffeng for continuous support and a plethora of helpful comments. I would also like to thank Germán Miranda for inspiration to connect this work to spectral theory.

\section{Examples}
\label{sec:examples}

In this section we give illustrative examples on how one could use the relative index theorem presented in \autoref{thm:relative_index_for_differential_operators}. In interest of a brief presentation we will omit details and manipulations outside the scope of this article.

\subsection{The classical relative index theorem}

\begin{example}
    \autoref{thm:computational_relative_index_theorem} generalizes the original relative index theorem found in \cite[Theorem 4.18]{Gromov_Lawson_1983} which we will now present. For \(i=1,2\), let \(\slashed{D}_i\) be a generalized Dirac operator constructed from a Clifford structure on a complete Riemannian manifold \(M_i\) (see \cite[Chapter 1]{Gromov_Lawson_1983}). Assume that there are compact sets \(K_i \subseteq M_i\) with smooth boundary such that \(M_1\setminus K_1\simeq M_2\setminus K_2\) preserving the Clifford structure, and hence \(\slashed{D}_1\) agree with \(\slashed{D}_2\) on \(M_i\setminus K_i\). Let \(\Sigma\) denote the common boundary of \(K_1\) and \(K_2\) and construct the closed manifold \(\widetilde{M}\coloneq K_1\cup_\Sigma (-K_2)\) by gluing \(K_1\) to \(K_2\) along \(\Sigma\). We endow \(\widetilde{M}\) with a Clifford structure by deforming the Clifford structures on \(K_1\) and \(K_2\) so that they are normal to \(\Sigma\) and reversing the Clifford orientation on \(K_2\). Then if \(\slashed{D}_1\) and \(\slashed{D}_2\) are invertible at infinity we have that
    \begin{equation}
        \label{eq:gromovlawson_relativeindextheorem}
        \ind \slashed{D}_{1,+} - \ind \slashed{D}_{2,+} = \ind \mathrlap{\widetilde{\phantom{D}}}\slashed{D}_+
    \end{equation}
    where \(\mathrlap{\widetilde{\phantom{D}}}\slashed{D}\) is the generalized Dirac operator constructed from the Clifford structure on \(\widetilde{M}\). See also the exposition in \cite[Theorem 1.1]{Bunke_1995}. 

    Note that \autoref{thm:relative_index_for_differential_operators} can also prove the \(\Phi\)-relative index theorem in \cite[Theorem 4.35]{Gromov_Lawson_1983}, which is the same as the relative index theorem except that the operators only have to agree on a connected component of the complement of a compact set.
\end{example}

\subsection{The splitting theorem}

\begin{example}
    \label{ex:splitting_formula}
\autoref{thm:relative_index_for_differential_operators} can be used to prove the splitting theorem for a Dirac operator (see for example \cite[Section 9]{Wahl_2007},\cite[Theorem 8.17]{Ballmann_Bar_2012} or  \cite[Theorem 5.9]{Bandara_2022}) which states that the index of a Dirac operator on a closed manifold is equal to the index of the same operator when cutting the manifold along a co-dimension one hypersurface and introducing Atiyah-Patodi-Singer boundary conditions. We will now present a sketch towards a generalization of the splitting theorem to elliptic operators of any order and arbitrary boundary conditions.

Let \(D\) be a Fredholm realization of an elliptic differential operator on a manifold \(M\) (perhaps with boundary and some appropriate boundary conditions on \(D\)) and let \(\Sigma\subseteq M\) be a compact two-sided hypersurface not intersecting \(\d M\). Let 
\[
    M_\Sigma=(M\setminus \Sigma) \cup (\Sigma_1\sqcup \Sigma_2)
\]
be \(M\) with \(\Sigma\) cut out and separated into two boundaries \(\Sigma_1,\Sigma _2\subseteq M_\Sigma\) diffeomorphic to \(\Sigma\), where \(\sqcup\) denotes disjoint union. Note that \(\d M_\Sigma \simeq \d M \sqcup \Sigma_1 \sqcup \Sigma_2\). The operator \(D\) can now also be seen as a differential operator on \(M_\Sigma\), but in order to have a well-behaved closed realization we need to impose boundary conditions at \(\Sigma_1\) and \(\Sigma_2\). We will here adopt the view of a boundary condition as a specified subspace of sections on the boundary. Specifically, let \(B_i\) be an appropriate subspace of sections on \(\Sigma_i\) and let \(D_{B_1\oplus B_2}\) denote the corresponding closed realization with 
\[
    \Dom D_{B_1\oplus B_2} = \setcond{f\in \Dom D}{\gamma f|_{\Sigma_i} \in  B_i \text{ for } i=1,2}
\]
where \(\gamma\) denotes the full trace map. For more on this way to view boundary conditions, see for example \cite{Bandara_Goffeng_Saratchandran_2023}.

We can now compare \(\ind D\) with \(\ind D_{B_1\oplus B_2}\) using \autoref{thm:relative_index_for_differential_operators}. To do this, we need to construct manifolds fitting a diagram such as \eqref{eq:the_diagram}. Fix a normal for \(\Sigma\subseteq M\) such that we obtain an inclusion \(\Sigma\times (-1,1)\into M\) for which we will call the second coordinate \(t\). This open set can also be embedded in \(\Sigma\times S^1\) where we can extend \(D|_{\Sigma\times (-1,1)}\), call it \(D_{S^1}\), by reflecting the symbol of \(D\) along \(t\). Similarly to the procedure constructing \(M_\Sigma\) from \(M\), we can cut out \(\Sigma\) form \(\Sigma\times S^1\) and obtain a manifold diffeomorphic to \(\Sigma\times [-1,1]\), where we similarly impose the same boundary conditions \(B_1, B_2\) on \(D_{S^1}\) to obtain an extension we call \(D_{[-1,1], B_1\oplus B_2}\). Now we can apply \autoref{thm:relative_index_for_differential_operators} to the operators \(D, D_{B_1\oplus B_2}, \ind D_{S_1}, D_{[-1,1], B_1\oplus B_2}\), where specifically we let \(U=\Sigma\times (-1,-\frac{1}{2})\sqcup \Sigma\times (\frac{1}{2},1)\) in \eqref{eq:the_diagram}. Using a local expression of the index of \(D_{S_1}\) with the Atiyah-Singer index theorem, the reflection can be made such that \(\ind D_{S_1}=0\) by symmetry. We obtain that
\[
    \ind D_{B_1\oplus B_2} = \ind D + \ind D_{[-1,1], B_1\oplus B_2},
\]
which is as far as we can go without looking at a specific operator and boundary conditions.

We can for example consider a first-order elliptic differential operator \(D\) such that \(D=c(\d_t + A)\) close to \(\Sigma\), where \(c\) is a vector bundle isomorphism and \(A\) is a differential operator on \(\Sigma\). Then if \(B_1\) and \(B_2\) are orthogonal and preserved by \(A\), then \(\ind \slashed{D}_{[-1,1],B_1\oplus B_2} = 0\) since then the solution operator \(c^{-1}e^{tA}\) preserves the boundary conditions. This situation arises for a Dirac operator if we assume that the metric is product close to \(\Sigma\) and consider APS-boundary condition and its dual. That is, 
\[
    \ind \slashed{D} = \ind \slashed{D}_{\text{APS}\oplus \text{APS\textsuperscript{dual}}}.
\]
\end{example}

\begin{example}
    \label{ex:splitting_z2}
    As the techniques used in this paper to prove the relative index theorem are quite general, one could look at other invariants than the Fredholm index behaving similarly. In particular, the techniques apply to the index invariants for real operators commuting with an action of the Clifford algebra \(Cl_k\) introduced in \cite{Atiyah_Singer_1969}, making it possible to extend \autoref{thm:relative_index_for_differential_operators} to an index theorem in real \(K\)-theory group \(KO_{p,q}(B)\). See for instance \cite{Schroder_1993} for more on real \(K\)-theory.

    For example, a splitting theorem for \emph{odd symmetric} Dirac operators in the secondary \(\bZ_2\)-valued index was recently presented in \cite{braverman2025gluingformulaz2valuedindex}. An odd symmetric operator is an operator \(D\) on a complex Hilbert space with an anti-unitary anti-involution \(\tau\) such that \(\tau D = D^* \tau\). Then one can consider the \(\bZ_2\)-valued index 
    \[
        \ind_\tau D = \dim \ker D \pmod{2}
    \]
    if \(D\) is Fredholm. This is equivalent to the case of real operators commuting with \(Cl_2\). In \cite{braverman2025gluingformulaz2valuedindex} they consider a manifold \(M\) with an involution \(\theta_M\colon M\to M\) and a similar setup as in \autoref{ex:splitting_formula} with the assumption that \(\Sigma\) is \(\theta_M\)-invariant and that there are neighborhoods on either side of \(\Sigma\) that are interchanged by \(\theta_M\). The involution \(\theta_M\) together with an anti-unitary anti-involution of a vector bundle induces an anti-unitary anti-involution \(\tau\) on sections. Then for a Dirac operator \(\slashed{D}\) on \(M\) that is odd symmetric with respect to \(\tau\), they prove the \(\bZ_2\)-valued splitting theorem
    \[
        \ind_\tau \slashed{D} = \ind_\tau \slashed{D}_{\text{APS}\oplus \text{APS\textsuperscript{dual}}}
    \]
    found as \cite[Theorem 4.8]{braverman2025gluingformulaz2valuedindex}. This result can also be proven using the discussion after \autoref{thm:relative_index_for_differential_operators} the same way as sketched in \autoref{ex:splitting_formula} if one assumes that the metric is product close to the boundary, although this assumption is only there for technical reasons.
\end{example}

\subsection{Extending index theorems to non-compact manifolds}

In the next examples, \autoref{ex:euclidian_at_infinity} and \autoref{ex:subLaplace_change_gamma}, we will illustrate how the relative index theorem can be used to extend index theorems from compact manifolds to non-compact manifolds.

\begin{example}
    \label{ex:euclidian_at_infinity}
    Let \(M\) be a non-compact (connected) manifold that is Euclidean at infinity in the sense that there is a compact set \(K\subseteq M\) such that \(M\setminus K\) is diffeomorphic to a finite disjoint union \(\bigsqcup_i \bR^n\setminus K_i\) for compact sets \(K_i\subseteq \bR^n\). Let \(D\) be a differential operator of order \(m\) on \(M\) acting on sections of vector bundles that are trivial on \(M\setminus K\). Assume that \(D\) is elliptic of order \(m\) at each point and that \(D\) restricted to \(\bR^n\setminus K_i\) can be extended a differential operator \(D_i\) on \(\bR^n\) (with trivial bundles) that is elliptic in the sense of Shubin \cite[Chapter IV]{Shubin_2001}. In particular, we can define the full-symbol \(\sigma_{\text{full}}(D)\) of \(D\) on \(M\setminus K\) which then satisfies
    \[
        \abs{\sigma_{\text{full}}(D)(x,\xi)} \gtrsim (\abs{x}^m + \abs{\xi}^m)I_k
    \]
    for all \((x,\xi) \in T^*(M\setminus K)\) outside a compact set, where \(\abs{x}\) is the Euclidean distance induced from \(\bR^n\) on \(M\setminus K\).

    A priori, one might ask if \(D\) is Fredholm. We will make use of \autoref{thm:computational_relative_index_theorem} to both show that \(D\) is Fredholm and calculate the index of \(D\) by referring to the relative index theorem and the Atiyah-Singer index theorem on \(\bR^n\) \cite{Fedosov_1970} (see also \cite[Section 19.3]{Hormander_2007} and \cite{Elliott_Natsume_Nest_1996}). Note that \(D_i\) is Fredholm for each \(i\) by \cite[Theorem 25.3]{Shubin_2001} and that \(D\) agree with \(\bigoplus_i D_i\) on \(M/K\simeq \bigsqcup_i \bR^n\setminus K_i\). Hence, using \autoref{thm:computational_relative_index_theorem} we obtain that \(D\) is Fredholm and 
    \begin{equation}\label{eq:euclidean_at_inf_index_diff}
        \ind D - \ind \bigoplus_i D_i = \int_{S^* K} \alpha(\sigma^m(D)) - \sum_i \int_{S^*K_i} \alpha(\sigma^m(D_i))
    \end{equation}
    where \(\sigma^m\) denotes the principal symbol and \(\alpha\) comes from the Atiyah-Singer index theorem. By choosing a \(K\subseteq M\) bigger, we can assume that \(K_i \subseteq \bR^n\) are balls and 
    \[\abs{\sigma_{\text{full}}(D)(x,0)} \gtrsim \abs{x}^m I_k\] 
    for all \(x\in M\setminus K\). Let \(\tilde{K}_i\subseteq T^*\bR^n\simeq \bR^{2n}\) denote the ball such that \(\tilde{K}_i\cap \bR^n = K_i\), then by the index formula in \cite{Fedosov_1970} we have that
    \begin{equation}\label{eq:euclidean_at_inf_index_Rn}
        \ind \bigoplus_i D_i = \sum_i \int_{\d \tilde{K}_i} \alpha(\sigma_{\text{full}}(D_i))
    \end{equation}
    where \(\alpha(\sigma) = \frac{(n-1)!}{(2n-1)! (2\pi i)^n} \Tr((\sigma^{-1}\rd\sigma)^{2n-1})\) explicitly. In combination with \eqref{eq:euclidean_at_inf_index_diff} we hence obtain an index formula for \(D\). To rework this formula further, note that \(\d \tilde{K}_i\) is homotopic to
    \[\d(B^*K_i)=S^*K_i \cup B^*K_i|_{\d K_i}\] 
    where \(B^*K_i\) denotes a ball bundle in \(T^*K_i\). In the index formula for \(D\) this means that the integral over \(S^*K_i\) in \eqref{eq:euclidean_at_inf_index_diff} will cancel against the \(S^* K_i\)-part of the integral in \eqref{eq:euclidean_at_inf_index_Rn} leaving only an integral over \(B^*K_i|_{\d K_i}\). But since \(\bigcup_i \d K_i = \d K\) we then obtain that the index of \(D\) can be expressed as an integral over \(S^*K \cup B^*K|_{\d K}=\d(B^*K)\). That is,
    \begin{equation}\label{eq:euclidean_at_inf_AB}
        \ind D = \int_{\d(B^*K)} \alpha(\tilde{\sigma}^m(D))
    \end{equation}
    where
    \[
    \tilde{\sigma}^m(D) = \begin{cases} 
        \sigma^m(D)(x,\xi) & (x,\xi) \in S^*K \\ 
        \sigma_{\text{full}}(D_i)(x,\xi) & (x,\xi) \in B^*K|_{\d K_i}
    \end{cases}\]
    is an invertible extension of \(\sigma^m(D)\). One may note the similarities between \eqref{eq:euclidean_at_inf_AB} and the Atiyah-Bott index formula \cite{Atiyah_Bott_1964} for elliptic operators with elliptic boundary conditions on compact manifolds with boundary. We expect that our assumption that \(D\) can be extended  from \(M/K\) to Shubin elliptic differential operators on \(\bigsqcup_i \bR^n\) imposes an elliptic boundary condition on \(D\) on \(K\) (possibly after choosing \(K\) bigger).
\end{example}

\begin{example}
    \label{ex:subLaplace_change_gamma}
    In this example we will consider the sub-Laplacian \(\Delta_H\) on contact manifolds and relate the index of \(\Delta_H + i\gamma Z\) on a compact contact manifold to that of a non-compact manifold, where \(Z\) is the Reeb vector field. 

    Let \(M\) be a contact manifold (without boundary), meaning a \((2n+1)\)-dimensional manifold with a specified co-dimension one subbundle \(H\subseteq TM\) such that \(H=\ker \theta\) for a \(1\)-form \(\theta\) with \(d\theta|_H\) non-degenerate. Note that for simplicity we are assuming that the contact structure is co-oriented meaning that \(\theta\) is globally defined. Let \(Z\) be the unique vector field such that \(\theta(Z)=1\) and \(d\theta(Z,\cdot) = 0\) called the Reeb vector field. By the Darboux theorem, there are locally defined vector fields \(X_i\) and \(Y_i\) spanning \(H\) at each point where these vector fields are defined such that \([X_i,Y_j] = \delta_{ij}Z\).
    
    We consider the horizontal differential \(d_H\colon C^\infty(M)\to C^\infty(M;T^*M)\) such that \(d_Hf(X)=X(f)\) for any vector field \(X\) with values in \(H\) and \(d_Hf(Z)=0\). Define the sub-Laplacian as
    \[
        \Delta_H = d_H^\dagger d_H
    \]
    where the adjoint is taken with respect to some fixed Riemannian metric on \(M\). The operator \(\Delta_H\) locally takes the form
    \[
        \Delta_H = \sum_i -X_i^2 - Y_i^2 + a_i X_i + b_i Y_i + c
    \]
    for functions \(a_i, b_i\) and \(c\).  Note that we use this specific construction of the sub-Laplacian to ensure that it is positive.

    We will consider an operator of the form
    \[
        D_\gamma = (\Delta_H + V)\otimes \id_k + iZ \gamma
    \]
    acting on \(C^\infty(M;\bC^k)\) where \(\id_k\) is the \(k\times k\)-identiy matrix, \(\gamma\in M_k(C^\infty(M))\) and \(V\in C^\infty(M)\) is such that \(V\geq \epsilon > 0\) outside a compact set. If \(M\) is a closed manifold then \(D_\gamma\) is Fredholm if and only if \(\gamma - \lambda \id_k\) is invertible for all \(\lambda \in (n + 2\bZ)\setminus (-n,n)\) and if so there is an explicit formula for the index \cite[Example 6.5.4]{Baum_van_Erp_2014}. We will show that if \(\gamma\) has compact support, then the same formula holds for non-compact manifolds.

    Let \(M\) be a non-compact contact manifold and assume that there is a compact \(K\subseteq M\) and a compact contact manifold \(\widetilde{M}\) such that \(K\) has a contact preserving embedding in \(\widetilde{M}\). Note that by \autoref{thm:characterization_of_Fredholm_geometric} we see that \(D_0\) on \(M\) is Fredholm with index zero since \(\Delta_H\) is (locally) Heisenberg-elliptic of order two and there is a \(k\in C_c^\infty(M)\) such that \(D_0 + k\) is strictly positive. Also, if \(\gamma\) has support in \(K\), then we can extend \(\gamma\) by zero to \(\widetilde{M}\). Denoting \(\tilde{D}_\gamma\) for \(D_\gamma\) on \(\widetilde{M}\), we can apply \autoref{thm:relative_index_for_differential_operators} to \(D_\gamma, D_0, \tilde{D}_\gamma\) and \(\tilde{D}_0\) to obtain that \(D_\gamma\) is Fredholm if and only if \(\gamma - \lambda \id_k\) is invertible for all \(\lambda \in (n + 2\bZ)\setminus (-n,n)\), and then
    \[
        \ind D_\gamma = \ind \tilde{D}_\gamma
    \]
    which can be calculated as in \cite[Example 6.5.4]{Baum_van_Erp_2014}.
\end{example}

\subsection{Isolating non-elliptic points}

\begin{example}
    \label{ex:singular_elliptic}
    In this example we will look at an operator that is elliptic at all but finitely many points. By cutting away these points and calculating their index contribution separately obtain an index formula as a sum over these points. This example is taken from \cite[Example 1.4]{mohsen2022indexmaximallyhypoellipticdifferential} where a considerably larger framework is constructed to deal with these type of situations.

    For a positive smooth function \(f\) on \(S^1\), consider the differential operators on \(S^1\times S^1\) of the form
    \[
        D_f = -\d_x^2 - f(x)\d_y^2 - i \gamma(x, y)\d_y
    \]
    where \(\gamma\in C^\infty(S^1\times S^1)\). If \(f(x)>0\) for all \(x\) then \(D_f\) is a Laplacian up to lower order terms and therefore has index zero. The operator we will consider will be \(D_{g^2}\) for some real-valued function \(g\in C^\infty(S^1)\) such that \(g'(x)\neq 0\) for all \(x\in g^{-1}(0)\). Note that this implies that \(g^{-1}(0)\) is finite. What we aim to do is to use \autoref{thm:relative_index_for_differential_operators} to cut out the problematic points in \(g^{-1}(0)\) in order to express \(\ind D_{g^2}\) as a sum over \(x\in g^{-1}(0)\) of indices of easier operators.
    
    Let \(U_x\) be small disjoint neighborhoods of each \(x\in g^{-1}(0)\) and let \(\chi_x\in C^\infty(S^1)\) be a positive function with \(\supp \chi_x \subseteq U_x\) and such that \(\chi_x(x)\neq 0\). Let \(\chi = \sum \chi_x\) and note that \(D_{g^2}\) and \(D_{g^2 + \chi}\) agree outside \(\bigcup U_x \times S^1\). Next, for each \(x\in g^{-1}(0)\) consider the differential operator on \(\bR\times S^1\) given by
    \[
        D_{x, f} = -\d_z^2 - f(z)\d_y^2 - i \gamma_x(z, y)\d_y
    \]
    for some positive function \(f\in C^\infty(\bR)\) and \(\gamma_x\in C_c^\infty(\bR\times S^1)\) is some function such that \(\gamma_x(z,y)=\gamma(x+z,y)\) for small \(z\). Let \(g_x\in C^\infty(\bR)\) be positive functions such that \(g_x(z)=g(x+z)\) for small \(z\) and \(g_x(z,y) = g'(x)z\) for large \(z\), then \(D_{g^2}\) matches \(D_{x,g_x^2}\) on \(U_x\times S^1\simeq V_x\times S^1\) for some neighborhood \(V_x\subseteq \bR\) of \(0\). Similarly, \(D_{g^2 + \chi}\) and \(D_{x,g_x^2 + \chi_x}\) matches on \(U_x\times S^1\simeq V_x\times S^1\). Also, \(D_{x,g_x^2}\) matches \(D_{x,g_x^2 + \chi_x}\) outside \(\supp{\chi_x}\times S^1\).

    Note that \(D_{g^2 + \chi}\) and \(D_{x, g_x^2 + \chi_x}\) are all Laplacians up to compactly supported lower order terms, and therefore are Fredholm with index zero. Hence, applying \autoref{thm:relative_index_for_differential_operators} to \(D_g^2, \bigoplus_x D_{x, g_x^2 + \chi_x}, D_{g^2+\chi}\) and \(\bigoplus_x D_{x, g_x^2}\) give us that \(D_{g^2}\) is Fredholm if (and only if) all of \(D_{x,g_x^2}\) are, and if so
    \[
        \ind D_{g^2} = \sum_{x\in g^{-1}(0)} \ind D_{x, g_x^2}.
    \]
    By choosing our neighborhoods \(U_x\) small enough we can ensure that \(g_x(z)\) is strictly increasing and \(g_x(z) - g'(x)z\) have small enough support so that we can perform a variable substitution \(y\mapsto w\) on \(S^1\) as \(g_x(z)y = g'(x)z w\). Then \(D_{x,g_x^2}\) takes the form
    \[
        H_{a,b} = -\d_z^2 - a^2z^2\d_w^2 - i b(z, w)\d_w
    \]
    where \(a=g'(x)\) and \(b(z,w) = \gamma_x(z,y)\), which is a perturbation of a series of Harmonic oscillators if one considers Fourier series in the variable \(w\). Although we will not do it here, it is feasible to examine Fredholmness and calculate the index of an operator as \(H_{a,b}\). Indeed, examining the formulas in \cite{mohsen2022indexmaximallyhypoellipticdifferential} we see that \(H_{a,b}\) is Fredholm if and only if \(b(0,y)\notin \pm a(2\bN - 1)\) for all \(y\in S^1\) and 
    \[
        \ind H_{a,b} = \sum_{\lambda \in a(2\bN - 1)} \wind(\lambda - b(0,\cdot)) - \wind(\lambda + b(0,\cdot)).
    \]
    Hence, using the work in this example we see \(D_{g^2}\) is Fredholm if and only if \(\gamma(x,y)\neq \pm g'(x)(2\bN-1)\) for all \((x,y)\in g^{-1}(0)\times S^1\) and if so
    \[
        \ind D_{g^2} = \sum_{x\in g^{-1}(0)} \sum_{\lambda \in g'(x)(2\bN - 1)} \wind(\lambda - \gamma(x,\cdot)) - \wind(\lambda + \gamma(0,\cdot)).
    \]

\end{example}

\subsection{Riemann surfaces and the Riemann-Roch theorem}

\begin{example}
    In this example we will look at Riemann surfaces and use the relative index theorem to reduce the proof of the Riemann-Roch theorem to the case of genus zero and one. We will only do this for the trivial line bundle, meaning we will calculate the difference in dimension between holomorphic functions and (anti-) holomorphic \(1\)-forms. In particular, this is the index of the \(\dbar\)-operator of a Riemann surface, which locally takes the form \(\frac{\d}{\d \bar{z}} d\bar{z}\) for a local analytic coordinate \(z\). We will show that
    \[
        \ind \dbar_g = 1 - g
    \]
    where \(\dbar_g\) denotes the \(\dbar\)-operator on a Riemannian surface \(\Sigma_g\) of genus \(g\geq 0\).

    We first give an explicit construction found in \cite{Seeley_Singer_1998} on how to add a handle (holomorphically) to a Riemann surface therefore increases its genus by one. For simplicity, we will assume that we can obtain any Riemann surfaces, including its analytic structure, inductively from a sphere this way. Let \(\Sigma\) be a Riemann surface and let \(\setcond{z}{\abs{z}<1}\) and \(\setcond{w}{\abs{w}<1}\) be two disjoint complex analytic charts in \(\Sigma\). If we cut out the center of radius \(\frac{1}{2}\) from these discs we can identify the inner boundary of the created annuli \(\setcond{z}{\frac{1}{2}\leq\abs{z}\leq 1}\) and \(\setcond{w}{\frac{1}{2}\leq \abs{w}\leq 1}\) using \(zw=\frac{1}{4}\).  Note that this identification is done holomorphically which ensures that we keep the structure of a Riemann surface, where specifically we have that \(\frac{d z}{z} = -\frac{dw}{w}\). 
    
    Using this procedure of adding a handle to two complex curves \(\Sigma_g\) and \(\Sigma_h\) we note that we can fit \(\Sigma_g,\Sigma_{g+1},\Sigma_h\) and \(\Sigma_{h+1}\) in a diagram like \eqref{eq:the_diagram} with \(U=\setcond{z}{\frac{3}{4}\leq\abs{z}\leq 1}\cup \setcond{w}{\frac{3}{4}\leq \abs{w}\leq 1}\). Since the \(\dbar\)-operator is preserved under holomorphic coordinate changes, we can apply \autoref{thm:relative_index_for_differential_operators} to \(\dbar_{g+1},\dbar_h,\dbar_g\) and \(\dbar_{h+1}\), and obtain
    \[
        \ind \dbar_{g+1} + \ind \dbar_h = \ind \dbar_g + \ind \dbar_{h+1}.
    \]
    Specifically, we can choose \(h=0\), and an induction argument shows that 
    \[
        \ind \dbar_g = \ind \dbar_1 + (g - 1)(\ind \dbar_1 - \ind \dbar_0)
    \]
    for \(g\geq 2\). What remains it to show that \(\ind \dbar_0=1\) and \(\ind \dbar_1=0\). We remind ourselves of Liouville's theorem, which says that any bounded  (anti-) holomorphic function on \(\bC\) is constant. Hence, if we view \(\Sigma_0\) as the Riemann sphere \(\bC \cup\set{\infty}\) and \(\Sigma_1\) as \(\bC\) quotient by a lattice, we see that \(1=\ker \dbar_0 = \ker \dbar_1\). Similarly, we obtain that \(\Coker \dbar_1=1\) since \(d\bar{z}\) can be globally defined on \(\Sigma_1\). Lastly, looking at how \(d\bar{z}\) transforms to \(-\frac{d\bar{z}}{\bar{z}^2}\) under the coordinate change \(z\to\frac{1}{z}\) on the Riemannian sphere \(\Sigma_0\), we can convince ourselves that there are no antiholomorphic \(1\)-forms on \(\Sigma_0\). Therefore, \(\Coker \dbar_0 = 0\) and we are done.
\end{example}

\section{The relative index theorem}
\label{sec:relative_index_theorem}

\subsection{Preliminaries for Hilbert modules and Fredholm operators}

A (right) Hilbert \(B\)-module \(\pE\) over a \Cst-algebra \(B\) is a Banach space with a right \(B\)-module structure such that the norm is induced by a \(B\)-valued inner product satisfying what one would expect from Hilbert spaces. More precisely, we denote the inner produce as \(\braket{\cdot, \cdot}\colon \pE\times \pE \to B\) which satisfies 
\begin{enumerate}
    \item \(\braket{x,y b} = \braket{x,y}b\);
    \item \(\braket{x,y+z} = \braket{x,y} + \braket{x,z}\);
    \item \(\braket{x,y} = (\braket{y,x})^*\);
    \item \(\braket{x,x}\geq 0\) (a positive element of \(B\));
    \item \(\norm{x}^2 = \norm{\braket{x,x}}_B\);
\end{enumerate}
for \(x,y,z\in \pE\) and \(b\in B\). For more on Hilbert modules, see for example \cite{Lance_1995} or \cite[Section 2.2]{Gracia-Bondia_Varilly_Figueroa_2001}. The concrete Hilbert modules we will be concerned with will be \(L^2\)-sections of a vector \(B\)-bundle, a bundle of finitely generated \(B\)-mobuldes. On such \(L^2\)-sections one can look at (pseudo-) differential operator with coefficients in \(B\), which was first considered by Mi{\v{s}}{\v{c}}enko-Fomenko in \cite{Miscenko_Fomenko_1979}. For Hilbert modules \(\pE,\pF\) we let \(\cL(\pE,\pF)\) (or \(\cL(\pE)\) if \(\pE=\pF\)) denote adjointable maps from \(\pE\) to \(\pF\), which automatically are \(B\)-linear and bounded. Note that not all \(B\)-linear bounded operators are adjointable. We equip \(\cL(\pE,\pF)\) with the operator norm and let \(\cK(\pE,\pF)\) (or \(\cK(\pE)\) if \(\pE=\pF\)) denote the closure of \(B\)-finite rank operators in \(\cL(\pE,\pF)\). Note that \(\cL(\pE)\) and \(\cK(\pE)\) are \Cst-algebras and \(\cK(\pE)\) is an ideal in \(\cL(\pE)\).

A Hilbert \((A,B)\)-bimodule \(\pE\) for \Cst-algebras \(A\) and \(B\) is a right Hilbert \(B\)-module with a representation of \(A\) in \(\cL(\pE)\). The representation of \(A\) can also be seen as a left action on \(\pE\), and we will therefore omit writing out the representation entirely. Note that when talking about operators on a Hilbert \((A,B)\)-bimodule \(\pE\) we typically mean operators on \(\pE\) as a right Hilbert \(B\)-module and do not assume any linearity with respect to \(A\).

We say that an operator \(T\in \cL(\pE,\pF)\) between Hilbert \(B\)-modules is left-\(B\)-Fredholm if there is an \(S\in \cL(\pF,\pE)\) such that \(1-ST \in \cK(\pE)\). The definition for right-\(B\)-Fredholm is analogous, and we say that \(T\) is \(B\)-Fredholm if it is both left- and right-\(B\)-Fredholm. Note that in the case when \(T\) is \(B\)-Fredholm such an \(S\) can be chosen to be the same in both cases. Any \(B\)-Fredholm operator \(T\) has a well-defined index \(\ind_B T\) as an element of the \(K\)-theory group \(K_0(B)\). We will refer the actual definition of the index to \cite[Chapter 4]{Gracia-Bondia_Varilly_Figueroa_2001}, only noting that it is independent of compact perturbations, composition with invertible elements and homotopy. If \(B=\bC\) these concepts coincides with the classical notion of Fredholm and Fredholm index on Hilbert spaces. In particular, an operator on a Hilbert space is left-Fredholm if and only if it has finite dimensional kernel and closed range, and zero is an isolated point in the spectrum of any Fredholm operator on a Hilbert space. This is in contrast to a general \(B\) since a \(B\)-Fredholm operator does not necessarily have closed range nor be finitely generated kernel.

For an unbounded operator \(D\) between Hilbert modules \(\pE\) and \(\pF\), we will use the notation \(D \colon \pE \dashrightarrow \pF\) signifying that \(D\) is defined on a \(B\)-submodule of \(\pE\) which we denote \(\Dom D\). Note that unbounded operators on Hilbert modules do not automatically satisfy all properties one would expect from Hilbert spaces. We will therefore often assume an operator \(D\) to be \emph{regular}, meaning that \(D\) is a densely defined closed operator with densely defined adjoint \(D^*\) such that \(1+D^*D\) has dense range. In particular, there is a continuous functional calculus for regular operators. See details in \cite[Chapter 9]{Lance_1995}, and \cite{Kaad_Lesch_2012} for another characterization of regular operators. We will need following lemma regarding regular operators.

\begin{lemma}
    \label{thm:closed_graph_thm_hilbert_modules}
    Let \(T\in \cL(\pE,\pF)\) and \(T\Dom D_1\subseteq \Dom D_2\) for two regular operators \(D_1\) and \(D_2\). Then \(T\) restricts to an element \(T\in \cL(\Dom D_1, \Dom D_2)\). 
\end{lemma}

\begin{proof}
    We will prove that \(S=(1+D_1^*D_1)^{-1}T^*(1+D_2^*D_2)\) extends to a map \(S\colon \Dom D_2\to \Dom D_1\) that is the adjoint of \(T\colon \Dom D_1 \to \Dom D_2\). For \(x\in \Dom D_1\) and \(y \in \Dom D_2^*D_2\) we have that
    \[
        \braket{x, Sy}_{D_1} = \braket{x, T^*(1+D_2^*D_2)y}_{\pE} = \braket{Tx, (1+D_2^*D_2)y}_{\pF} = \braket{Tx,y}_{D_2}
    \]
    where the subscript denotes which inner product it is. Since \(\Dom D_2^*D_2\) is a core of \(D_2\) and \(T\colon \Dom D_1 \to \Dom D_2\) is continuous by the closed graph theorem, we can extend \(S\) to a map \(S\colon \Dom D_2\to D_1\).
\end{proof}

For a regular unbounded operator \(D\) we say that \(D\) is \(B\)-Fredholm if \(D\) is \(B\)-Fredholm as an operator in \(\cL(\Dom D, \pE)\). This is equivalent to \(D(1+D^*D)^{-\frac{1}{2}}\) being \(B\)-Fredholm. Note that if \(D\subseteq D_e\) and \(D_e\) is left-\(B\)-Fredholm, then \(D\) is left-\(B\)-Fredholm. This can be seen since \(\Dom D\) is a closed submodule of \(\Dom D_e\).

\subsection{Relatively compact perturbations}

The relative index theorem in \autoref{thm:relative_index_for_differential_operators} will follow from a quite simple abstract result presented in this section together with a clever construction due to Bunke \cite{Bunke_1995}. What will be used is \emph{relatively compact perturbations}, meaning that the difference of two operators on a core of both operators extends to a compact operator from one of the operators domains. This implies that the operators have the same domain and the same Fredholm properties. To prove this we only need the following lemma.

\begin{lemma}
    \label{thm:upper_bound_with_compact_operator}
    Let \(D\colon \pE \dashrightarrow \pF\) be a regular operator on Hilbert \(B\)-modules, and \(T\colon \pE \dashrightarrow \pF'\) an operator such that \(T(1+D^*D)^{-\frac{1}{2}}\) can be extended to a \(B\)-compact operator in \(\cK(\pE,\pF')\), then for any \(\epsilon>0\) there is an \(M>0\) such that
    \[
        \norm{Tx}\leq M\norm{x} + \epsilon\norm{Dx}
    \]
    for \(x\in \Dom D\cap \Dom T\).
\end{lemma}

\begin{proof}
    Fix \(\epsilon>0\) and fix a finite rank operator \(K=\sum_k \ketbra{a_k}{b_k}\) such that \(\norm{T(1+D^*D)^{-\frac{1}{2}} - K}<\frac{1}{\sqrt{2}}\epsilon\) (or to be the precise we consider the extension of \(T(1+D^*D)^{-\frac{1}{2}}\)). By density of \(\Dom D\) we can also assume that \(b_k\in \Dom D\). Then for any \(x\in \Dom D\cap \Dom T\)
    \[
        \begin{split}
            \norm{Tx}
            &= \norm{T(1+D^*D)^{-\frac{1}{2}}(1+D^*D)^{\frac{1}{2}}x} \\
            &\leq \norm{K(1+D^*D)^{\frac{1}{2}}x} + \tfrac{1}{\sqrt{2}}\epsilon\norm{(1+D^*D)^{\frac{1}{2}}x} \\
            &\leq \left(\epsilon + \sum \norm{a_k}\norm{(1+D^*D)^{\frac{1}{2}}b_k}\right)\norm{x} + \epsilon\norm{Dx}. \qedhere
        \end{split}
    \]
\end{proof}

\begin{lemma}
    \label{thm:relativly_compact_pert_domains}
    Let \(D\colon \pE \dashrightarrow \pF\) and \(D'\colon \pE' \dashrightarrow \pF'\) be regular operators between Hilbert \(B\)-modules. Assume that \(U\in \cL(\pE, \pE')\) and \(V\in \cL(\pF, \pF')\) are invertible maps such that 
    \begin{enumerate}
        \item \(U\) maps a core \(\cE\) of \(D\) onto a core of \(D'\);
        \item \((D'U-VD)(1+D^*D)^{-\frac{1}{2}}\) extends to a compact operator in \(\cK(\pE, \pF')\).
    \end{enumerate}
    Then \(U\) restricts to a bounded and invertible element \(U\in \cL(\Dom D,\Dom D')\).

    Furthermore, \(D'\) is regular.
\end{lemma}

\begin{proof}
    For \(x\in \cE\)
    \[
        \norm{D'Ux}\leq \norm{(D'U - VD)x} + \norm{V} \norm{Dx}
    \]
    which since \(D'U-VD \colon \Dom D\to \pF'\) is bounded means that
    \[
        \norm{Ux}_{D'}\lesssim \norm{x}_D.
    \]
    On the other hand, using that \(D'U-VD \colon \Dom D\to \pF'\) is compact there is an \(M>0\) such that for any \(x\in \cE\)
    \[
        \begin{split}
        \norm{Dx} 
        &\leq \norm{V^{-1}}\norm{(D'U - VD)x} + \norm{V^{-1}}\norm{D'Ux} \\
        &\leq M \norm{V^{-1}} \norm{U^{-1}} \norm{Ux} + \tfrac{1}{2}\norm{Dx} + \norm{V^{-1}}\norm{D'Ux}
        \end{split}
    \]
    meaning that
    \[
        \norm{x}_D \lesssim \norm{Ux}_{D'}.
    \]
    Using that \(\cE\) is a core of \(D\) and \(U \cE\) is a core of \(D'\), we can  extend \(U\) to a bijective map \(U\colon \Dom D \to \Dom D'\) by continuity and from \autoref{thm:closed_graph_thm_hilbert_modules} we obtain that \(U\) is adjointable.
\end{proof}

Note that \autoref{thm:relativly_compact_pert_domains} is symmetric in the sense that \((D'U-VD)(1+D^*D)^{-\frac{1}{2}}\) extends to a compact operator if and only if \((DU^{-1}-V^{-1}D')(1+D'^*D')^{-\frac{1}{2}}\) does.

\begin{lemma}
    \label{thm:relativly_compact_pert_Fredholm}
    Let \(D, D', U\) and \(V\) be as in \autoref{thm:relativly_compact_pert_domains}, then \(D\) is \(B\)-Fredholm if and only if \(D'\) is, and if so \(\ind_B D=\ind_B D'\).
\end{lemma}

\begin{proof}
    Since \(V^{-1}D'U - D \in \cK(\Dom D, \pE)\) we have that if \(D\) is Fredholm, the so is \(V^{-1}D'U\) and then also \(D'\). In particular,
    \[
        \ind D = \ind V^{-1}D' U = \ind D'. \qedhere
    \]
\end{proof}

In the context of pseudodifferential calculi, one should see previous lemma as saying that Fredholm index of an elliptic operator in any pseudodifferential calculus is independent of compactly supported lower order perturbations if the pseudodifferential calculus has an analog to Rellich theorem, that is, \(H^s(K)\into L^2(K)\) is compact for any \(s>0\) if \(K\) is compact.

\subsection{Bunke's unitary}
\label{sec:bunkesunitary}

In this section we will take a unitary constructed in \cite{Bunke_1995} and combine it with \autoref{thm:relativly_compact_pert_Fredholm} to prove \autoref{thm:relative_index_for_differential_operators}.

Consider the manifolds \(M_{ij}\) from the diagram \eqref{eq:the_diagram}. Let \(\chi_{V_i} \in C^\infty(V_i)\) and \(\chi_{W_j}\in C^\infty(W_j)\) for \(i,j=1,2\) be positive functions such that
\[
    \chi_{V_i}^2 + \chi_{W_j}^2 = 1 \text{ in } C^\infty(M_{ij}).
\]
Note that to construct such functions one only needs to choose one, say \(\chi_{V_1}\in C^\infty(M_{11})\) such that \(\chi_{V_1} = 0\) on \(M_{11}/V_1\) and \(\chi_{V_1} = 1\) on \(M_{11}/W_1\). In particular, \(\chi_{V_1} = \chi_{V_2}\) on \(U\) and similarly for \(\chi_{W_i}\). In what follows, we will also see \(\chi_{V_i}\) and \(\chi_{W_i}\) as positive contractive maps in the diagram
\[
    \begin{tikzcd}
        L^2(M_{11}) \ar{d}{\chi_{W_1}} \ar{r}{\chi_{V_1}} & L^2(M_{12}) \\
        L^2(M_{21}) & L^2(M_{22}) \ar{u}{\chi_{W_2}} \ar{l}{\chi_{V_2}}
    \end{tikzcd}
\]

\begin{lemma}
    \label{thm:bunkes_unitary}
    The map
    \[
        \pmat{
            \chi_{V_1} & -\chi_{W_2} \\
            \chi_{W_1} & \chi_{V_2}
        }
        \colon L^2(M_{11})\oplus L^2(M_{22}) \to L^2(M_{12})\oplus L^2(M_{21})
    \]
    is a unitary.
\end{lemma}

\begin{proof}
    To see that the map is unitary, we calculate that
    \[
        \pmat{
            \chi_{V_1} & -\chi_{W_2} \\
            \chi_{W_1} & \chi_{V_2}
        }^* \pmat{
            \chi_{V_1} & -\chi_{W_2} \\
            \chi_{W_1} & \chi_{V_2}
        } = \pmat{
            \chi_{V_1}^2 + \chi_{W_1}^2
            & -\chi_{V_1}\chi_{W_2} + \chi_{W_1}\chi_{V_2} \\
            -\chi_{W_2}\chi_{V_1} + \chi_{V_2}\chi_{W_1}
            & \chi_{W_2}^2 + \chi_{V_2}^2 
        }
    \]
    where the off-diagonal elements are zero since \(\chi_{V_1}\chi_{W_2}\) has support on \(U\) and \(\chi_{V_1}\chi_{W_2} = \chi_{V_2}\chi_{W_1}\) on \(U\), and the diagonal elements are one by construction.
\end{proof}

\begin{proof}[Proof of \autoref{thm:relative_index_for_differential_operators}]
    Recall the setup presented before \autoref{thm:relative_index_for_differential_operators}. Consider the unitary from \autoref{thm:bunkes_unitary} both as a map
    \[
        \label{eq:invertible_in_relative_index_proof}
        \pmat{
            \chi_{V_1} & -\chi_{W_2} \\
            \chi_{W_1} & \chi_{V_2}
        }
        \colon L^2(M_{11};E_{11})\oplus L^2(M_{22};E_{22}) \to L^2(M_{12};E_{12})\oplus L^2(M_{21};E_{21})
    \]
    and the analogous map for the \(F_{ij}\) bundles. Note that the map \eqref{eq:invertible_in_relative_index_proof} maps \(\cE_{11}\oplus \cE_{22}\) onto \(\cE_{12}\oplus \cE_{21}\) by assumption. With these invertible maps we will show that \(\pmat{D_{12} & \\ & D_{21}}\) is a relatively compact perturbation of \(\pmat{D_{11} & \\ & D_{22}}\) and apply \autoref{thm:relativly_compact_pert_Fredholm} to obtain our result. Acting on \(\cE_{11}\oplus \cE_{22}\)
    \[
        \begin{split}
        & \pmat{D_{12} & \\ & D_{21}}
        \pmat{
            \chi_{V_1} & -\chi_{W_2} \\
            \chi_{W_1} & \chi_{V_2}
        } - \pmat{
            \chi_{V_1} & -\chi_{W_2} \\
            \chi_{W_1} & \chi_{V_2}
        } \pmat{D_{11} & \\ & D_{22}} \\
        &= \pmat{
            D_{12} \chi_{V_1} - \chi_{V_1} D_{11}  & -D_{12} \chi_{W_2} + \chi_{W_2} D_{22} \\
            D_{21} \chi_{W_1} - \chi_{W_1} D_{11} & D_{21}\chi_{V_2} - \chi_{V_2} D_{22}
        } \\
        &=
        \pmat{
            [D_{11}, \chi_{V_1}] & -[D_{22}, \chi_{W_2}] \\
            [D_{11}, \chi_{W_1}] & [D_{22}, \chi_{V_2}]
        } + \pmat{
            (D_{12} - D_{11}) \chi_{V_1} & -(D_{12} - D_{22}) \chi_{W_2} \\
            (D_{21} - D_{11})\chi_{W_1} & (D_{21} - D_{22})\chi_{V_2}
        }
        \end{split}
    \]
    which we claim extends to a compact operator from \(\Dom D_{11}\oplus \Dom D_{22}\).  
    
    The operators of the form \([D_{ij}, \chi]\) have support in \(U\) since all \(\chi\):s are constant outside \(U\), and then we can utilize the trick that if \(a\in C_c^\infty(U)\) such that \(a = 1\) where \(\chi\) is not constant, then 
    \[
        [D_{ij}, \chi] = [D_{ij}, \chi]a = [D_{ij}, \chi a] - \chi[D_{ij}, a]
    \]
    which extends to a relatively compact operator by assumption. The operators of the form \((D_{ij} - D_{lk}) \chi\) are zero by assumption.

\end{proof}

\section{A characterization of Fredholmness}
\label{sec:characterization_of_Fred}

In \cite{Bunke_1995} the Dirac operators \(D\) looked at were assumed to be \emph{invertible at infinity} in the sense that there is an \(f\in C_c^\infty(M)\) such that \(f + D^2\) is invertible (with bounded inverse). Invertibility at infinity ensures that the operator is Fredholm, but we will show that the reverse is also true for operators satisfying a weak regularity assumption, making invertibility at infinity a characterization of Fredholmness. We also look at the notion of \emph{coercive at infinity} found in \cite[Definition 8.2]{Ballmann_Bar_2012} and \cite[Definition 3.1]{Bandara_2022}, which is equivalent if additional assumptions are made.

We will start of in the abstract showing that an analog to the \(f\) always exists for left-\(B\)-Fredholm operators on a right \(B\)-module. Then we will show that this element can be replaced by an element in any approximate unit, which in the geometric setting can be chosen to be in \(C_c^\infty(M)\).

\begin{proposition}
    \label{thm:characterization_of_Fredholm_abstract}
    For an unbounded regular operator \(D\colon \pE\dashrightarrow \pF\) between Hilbert \(B\)-modules the following are equivalent:
    \begin{enumerate}
        \item \(D\) is left-\(B\)-Fredholm;
        \item There is a self-adjoint \(K\in \cL(\pE)\cap\cK(\Dom D, \pE)\) such that \(K + D^*D\) is invertible;
        \item There is a \(K\in \cL(\pE, \pG)\cap \cK(\Dom D, \pG)\) for some Hilbert \(B\)-module \(\pG\) such that
        \[
            \norm{x}\lesssim \norm{Kx} + \norm{Dx}
        \]
        for all \(x\) in a core of \(D\);
        \item \label{enum:weylsequence} There exists no so-called \emph{Weyl sequence} for \(D\), that is, there is no sequence \(\set{x_n}\subseteq \Dom D\) such that \(\norm{x_n}=1, \norm{Dx_n}\to 0\) and \(\braket{y,x_n}\to 0\) for all \(y\in \pE\).        
    \end{enumerate}
\end{proposition}

\begin{proof}

    To show that these statements are equivalent we will prove the implications \(1\Rightarrow 2\Rightarrow 3\Rightarrow 1\) and \(1 \Leftrightarrow 4\).

    Towards showing that \(1\Rightarrow 2\) and \(1\Rightarrow 4\), assume that \(D\) is left-\(B\)-Fredholm. Note that span of \(B\)-finite rank operators of the form \(\sum_k^n \ketbra{l_k}{r_k}_D\) where \(l_k\in \Dom D\) and \(r_k\in \Dom D^*D\) is a dense left-ideal \(J\ideal\cK(\Dom D)\) since \(\Dom D^*D\) is a core of \(\Dom D\). Here we use the notation \(\ketbra{x}{y}_D\) for the operator \(\ketbra{x}{y}_D z = x(\braket{y, z}_{\pE} + \braket{Dy,Dz}_{\pE})\). In particular, \(D\) is left-invertible modulo \(J\) following the proof of \cite[Lemma 4.4]{Gracia-Bondia_Varilly_Figueroa_2001}. That is, if \(S\in \cL(\pE,\Dom D)\) is such that \(1-SD\in \cK(\Dom D)\), then there is a \(j\in J\) such that \(\norm{1-SD-j}_{\Dom D}<1\). Then \(SD+j\in \cK(\Dom D)\) is invertible and \(S'=(SD+j)^{-1}S\in J\) satisfies 
    \[
        1 - S'D = 1 - (SD + j)^{-1}SD = (SD + j)^{-1}j \in J.
    \]
    Now, fix an \(S\in \cL(\pE, \Dom D)\) and \(K=\sum_k^n \ketbra{l_k}{r_k}_D\) such that \(1-SD = K\) and \(r_k\in \Dom D^*D\). Then \(K = \sum_k^n \ketbra{l_k}{(1+D^*D)r_k}\) in the usual \emph{ketbra}-notation from \(\cL(\pE)\), meaning that \(K\) extends to a \(B\)-finite rank operator in \(\cL(\pE)\) as well. In particular, \(K\in \cK(\pE)\cap \cK(\Dom D)\subseteq \cL(\pE)\cap\cK(\Dom D, \pE)\). Finally, for any \(x\in \Dom D\) we obtain that
    \[
        \norm{x} = \norm{(K+SD)x}\leq \norm{Kx} + \norm{S}\norm{Dx}
    \]
    where we use \(S\in \cL(\pE,\Dom D)\subseteq \cL(\pE)\). Also,
    \[
        \norm{x}\leq \norm{Kx} + \norm{S}\norm{Dx} \leq \sum_k^n\norm{l_k}\norm{\braket{(1+D^*D)r_k, x}} + \norm{S}\norm{Dx}
    \]
    which contradicts existence of a Weyl sequence.

    To show that \(2\Rightarrow 3\), assume there is a \(K\in \cL(\pE,\pG)\cap \cK(\Dom D, \pG)\) such that \(\norm{x}\lesssim \norm{Kx} + \norm{Dx}\). Then \(K^*K + D^*D\) is invertible since it is bounded from below and dense domain \(\Dom D^*D\). Here \(K^*\) is the adjoint in \(\cL(\pE, \pG)\), and we note that \(K^*K\in \cL(\pE)\cap\cK(\Dom D, \pE)\).

    To show that \(3\Rightarrow 1\), assume there is a \(K\in \cL(\pE)\cap \cK(\Dom D, \pE)\) such that \(K+D^*D\) is invertible. Let \(S = (D(K+D^*D)^{-1})^*\) where the adjoint is taken as an operator in \(\cL(\pE,\pF)\), then on \(\Dom D^*D\), which is dense in \(\Dom D\), we have that
    \[
        SD = (K+D^*D)^{-1}D^*D = 1 - (K+D^*D)^{-1}K.
    \]
    Since \((K+D^*D)^{-1}K\in \cK(\Dom D)\) we have shown that \(D\) is left-\(B\)-Fredholm.

    To show that \(3\Rightarrow 1\) we shall instead assume that for all \(K\in \cL(\pE, \pG)\cap \cK(\Dom D, \pG)\) and all \(\epsilon>0\) there is an \(x\in \Dom D\) such that \(\epsilon \norm{x}\geq \norm{Kx}+\norm{Dx}\) and construct a Weyl sequence. Let \(K_n\in \cK(\pE)\) be a self-adjoint approximate unit for the \Cst-algebra \(\cK(\pE)\), that is, \(K_n a \to a\) for any \(a\in \cK(\pE)\). Note that \(\cK(\pE)\subseteq \cL(\pE)\cap \cK(\Dom D, \pE)\), so by assumption we can for each \(n=1,2,\dots\) choose an \(x_n\in \Dom D\) such that \(\norm{x_n}=1, \norm{Dx_n} < \frac{1}{n}\) and \(\norm{K_n x_n} < \frac{1}{n}\). To show that \(\set{x_n}\) is a Weyl sequence, fix a \(y\in \pE\), then
    \[
        \begin{split}
            \norm{\braket{y, x_n}}^2 
            &= \norm{\braket{y\braket{y,x_n}, x_n}} \\
            &\leq\norm{\braket{(\ketbra{y}{y} - K_n \ketbra{y}{y}) x_n, x_n}} + \norm{\braket{K_n y \braket{y, x_n}, x_n}} \\
            &\leq \norm{\ketbra{y}{y} - K_n \ketbra{y}{y}} + \norm{y}~\norm{\braket{y, x_n}}~\norm{K_n x_n}
        \end{split}
    \]
    and hence
    \[
        \begin{split}
            \norm{\braket{y, x_n}}
            &\leq \frac{1}{2}\norm{y}~\norm{K_n x_n} + \sqrt{\norm{\ketbra{y}{y} - K_n \ketbra{y}{y}} + \frac{1}{4}\norm{y}^2~\norm{K_n x_n}^2}
        \end{split}
    \]
    which goes to zero with \(n\).
\end{proof}

In the case of Banach spaces a related statement to parts of \autoref{thm:characterization_of_Fredholm_abstract} can be found in \cite[Proposition A.3]{Ballmann_Bar_2012}. Note also that for a Hilbert space, a natural choice of \(K\) would be \(K=P_{\ker D}\). This does not work as well for Hilbert modules as then Fredholmness does not necessarily mean isolated spectrum at zero and \(\ker D\) might not be finitely generated.

The criterion regarding Weyl sequences in \autoref{thm:characterization_of_Fredholm_abstract} is only included as a curiosity and will not be used in the paper. It relates to Weyl's criterion \cite[Theorem 25.57]{Gustafson_Sigal_2020} which says that a \(\lambda\in \bR\) is in the essential spectrum of a self-adjoint operator \(D\) if and only if \(D-\lambda\) admits a Weyl sequence. \autoref{thm:characterization_of_Fredholm_abstract} hence extends this classical result from Hilbert spaces to Hilbert modules.

For later use in relation to truly unbounded Kasparov modules, we introduce the notation
\begin{equation}
    \label{eq:compinv_def}
    \begin{split}
        \CompInv(D) \coloneq \{\, k \in \cL(\pE) \,:\,~ 
        & k \Dom D \subseteq \Dom D,~k(1+D^*D)^{-\frac{1}{2}} \in \cK(\pE) \\
        & k \text{ is positive and } k + D^*D \text{ is invertible}  \,\}
    \end{split}
\end{equation}
for an unbounded regular operator \(D\colon \pE\dashrightarrow \pF\). Note that saying that \(1\in \CompInv(D)\) is the same as \(D\) having compact resolvent.

\begin{corollary}
    \label{thm:compinv}
    Let \(D\colon \pE\dashrightarrow \pF\) be an unbounded regular operator, then \(D\) is left-\(B\)-Fredholm if and only if \(\CompInv(D)\) is non-empty.
\end{corollary}

\begin{proof}
    If \(D\) is left-\(B\)-Fredholm then we can modify procedure in the proof of \autoref{thm:characterization_of_Fredholm_abstract} and further requiring \(r_i\in \Dom (D^*D)^2\) instead of only \(\Dom D^*D\) which also is a core of \(\Dom D\). This ensures that \(K=\sum_k^n \ketbra{l_k}{(1+D^*D)r_k}\) for \(l_k,r_k\in \Dom D\) such that \((1+D^*D)r_k\in \Dom D\) and there is an \(S\in \cL(\pE,\Dom D)\) such that \(K = 1-SD\). In particular, \(K^*K \in \CompInv(D)\). The reverse implication is a direct consequence of \autoref{thm:characterization_of_Fredholm_abstract}.
\end{proof}

In the geometric case of an operator on \(L^2\)-sections as in \autoref{thm:characterization_of_Fredholm_geometric} and \ref{thm:characterization_of_Fredholm_geometric_extended}, we want to choose the \(K\) in \autoref{thm:characterization_of_Fredholm_abstract} to be an element of \(C_c^\infty(M)\). To accomplish this we will use strict approximate units. A strict approximate unit \(\set{\chi_n}\subseteq \cL(\pE)\) is a sequence strictly converging to the identity in the sense that
\[\norm{x-\chi_n x}\to 0 \text{ for any } x\in \pE.\] 
In particular, \(C_c^\infty(M)\) acting on \(L^2\)-sections contains strict approximate units. Note that a strict approximate unit is a weaker condition than an approximate unit, meaning a sequence \(\set{\chi_n'}\subseteq A\) in a \Cst-algebra \(A\) such that \(\norm{a-\chi_n'a}\to 0\) for any \(a\in A\), and the word strict is often replaced with strong when talking about Hilbert spaces instead of Hilbert modules.

\begin{lemma}
    \label{thm:finit_rank_and_approx_unit}
    Let \(K\in \cL(\pE)\) be a \(B\)-finite rank operator and \(\set{\chi_n} \subseteq \cL(\pE)\) a self-adjoint strict approximate unit. Then for any \(\epsilon>0\) there is an \(N\) such that \(\norm{K(1-\chi_n)} \leq \epsilon\) for all \(n\geq N\).
\end{lemma}

\begin{proof}
    Fix an \(\epsilon > 0\). Let \(l_k, r_k\in \pE\) be such that \(K=\sum_k^n \ketbra{l_k}{r_k}\) and fix an \(N\) such that \(\norm{(1-\chi_n) r_k} < \epsilon (\sum_k^n \norm{l_k})^{-1}\) for all \(j=1,\dots,n\) and \(n\geq N\), then for any \(x\in \pE\)
    \[
        \begin{split}
            \norm{K(1-\chi_n)x}
            &= \norm*{\sum_k^n l_k \braket{r_k, (1-\chi_n)x}} = \norm*{\sum_k^n l_k \braket{(1-\chi_n)r_k, x}} \\
            &\leq \sum_k^n \norm{l_k} \norm{(1-\chi_n)r_k} \norm{x} < \epsilon \norm{x}. \qedhere
        \end{split}
    \]
\end{proof}

\begin{proposition}
    \label{thm:approx_unit_iff_Fred}
    Let \(D\colon \pE\dashrightarrow \pF\) be an unbounded regular operator and \(\set{\chi_n} \subseteq \cL(\pE)\) a self-adjoint strictly approximate unit such that \(\chi_n(1+D^*D)^{-\frac{1}{2}}\in \cK(\pE)\) for all \(n\), then \(D\) is left-\(B\)-Fredholm if and only if there is an \(n\) such that
    \[
        \norm{x} \lesssim \norm{\chi_n x} + \norm{Dx}.
    \]
\end{proposition}

\begin{proof}
    Assume that \(D\) is left-\(B\)-Fredholm, then as in the proof \autoref{thm:characterization_of_Fredholm_abstract} we can find a \(B\)-finite rank \(K\in\cL(\pE)\) operator and an \(S\in \cL(\pE, \Dom D)\) such that \(K=1-SD\). Using \autoref{thm:finit_rank_and_approx_unit}, let \(\chi_n\) be such that \(\norm{K(1-\chi_n)} < 1\), then    
    \[
            \norm{x} \leq \norm{Kx} + \norm{S}\norm{D x} \leq \norm{K(1 - \chi_n)} \norm{x} + \norm{K}\norm{\chi_n x} + \norm{S}\norm{D x}
    \]
    from which we obtain the sought estimate.

    The implication in the other direction is a direct consequence of \autoref{thm:characterization_of_Fredholm_abstract}.
\end{proof}

Turning to the geometric case, note that \(C_c^\infty(M)\) acting on \(L^2(M;E)\) for a vector \(B\)-bundle contains (strict) approximate units. Also, if we have that \(\Dom D\subseteq H_{\loc}^s(M;E)\) for an unbounded regular operator \(D\), then \(k(1+D^*D)^{-\frac{1}{2}}\) is compact for any \(k\in C_c^\infty(M)\) since inclusion of Fréchet spaces automatically are always continuous (by the closed graph theorem) and 
\[
    \Dom D \into H_{\loc}^s(M;E) \overset{k}{\to} H^s(\supp{\chi};E) \into L^2(M;E)
\]
where the last map is compact by the Rellich theorem \cite[Theorem 7.4]{Shubin_2001}. Hence, \autoref{thm:approx_unit_iff_Fred} is an abstract version of \autoref{thm:characterization_of_Fredholm_geometric}.

For \autoref{thm:characterization_of_Fredholm_geometric_extended} we need to control the norm of a commutator, which we will do with the following lemma.

\begin{lemma}
    \label{thm:compact_upper_bound_with_approx_unit}
    Let \(D\colon \pE\dashrightarrow \pF\) be an unbounded regular operator, \(T\in \cK(\Dom D, \pE)\) and \(\set{\chi_n} \subseteq \cL(\pE)\) a self-adjoint strictly approximate unit, then for any \(\epsilon>0\) there is an \(M>0\) and an \(N\) such that 
    \[
        \norm{Tx} \leq \epsilon\norm{x} + \epsilon\norm{Dx} + M\norm{\chi_n x}
    \]
    for all \(x\in \Dom D\) and for all \(n\geq N\).
\end{lemma}

\begin{proof}
    Fix \(\epsilon>0\) and let \(K\) be a \(B\)-finit rank operator such that \(\norm{T(1+D^*D)^{-\frac{1}{2}} - K}<\epsilon\). Let \(l_k, r_k\in \pE\) be such that \(K=\sum_k^n \ketbra{l_k}{r_k}\). We can assume that \(r_k\in \Dom D\) by density. Now \(K(1+D^*D)^{\frac{1}{2}} = \sum_k^n \ketbra{l_k}{(1+D^*D)^{-\frac{1}{2}}r_k}\) and extends to a \(B\)-finite rank operator on \(\pE\). Using \autoref{thm:finit_rank_and_approx_unit}, fix an \(N\) such that \(\norm{K(1+D^*D)^{-\frac{1}{2}}(1-\chi_n)} < \epsilon\). Then for any \(x\in \Dom D\)    
    \[
        \begin{split}
            \norm{Tx}
            &= \norm{T(1+D^*D)^{-\frac{1}{2}}(1+D^*D)^{\frac{1}{2}}x} \\
            &\leq \norm{K(1+D^*D)^{\frac{1}{2}}x} + \epsilon\norm{(1+D^*D)^{\frac{1}{2}} x} \\
            &\leq \epsilon\norm{x} + \norm{K(1+D^*D)^{\frac{1}{2}}}\norm{\chi_nx} + \sqrt{2}\epsilon\norm{x} + \sqrt{2}\epsilon\norm{Dx}
        \end{split}
    \]
    which after a rescaling of \(\epsilon\) is the sought estimate.
\end{proof}

\begin{proposition}
    \label{thm:upper_bound_when_restricted_by_approx_unit}
    Let \(D\colon \pE\dashrightarrow \pF\) be an unbounded regular operator on Hilbert \((A,B)\)-bimodules such that \(\cl{A \pE} = \pE\) and \(\cA\subseteq A\) a dense \(*\)-subalgebra such that for any \(a\in \cA\) there is an \(a'\in \cA\) such that \(a=aa'\). Assume additionally for all \(a\in \cA\) that 
    \begin{enumerate}
        \item \(a\Dom D\subseteq \Dom D\);
        \item \(a(1+D^*D)^{-\frac{1}{2}}\in \cK(\pE)\); 
        \item \([D,a](1+D^*D)^{-\frac{1}{2}}\in \cK(\pE,\pF)\),
    \end{enumerate}
    then \(D\) is left-\(B\)-Fredholm if and only if there is a \(\chi\in \cA\) such that 
    \[
        \norm{(1-\chi)x}\lesssim \norm{D(1-\chi)x}
    \]
    for all \(x\in \Dom D\).
\end{proposition}

\begin{proof}

    Assume \(D\) is left-\(B\)-Fredholm. Since there is an approximate unit in \(\cA\) that is also an approximate unit in \(\cL(A)\), \autoref{thm:approx_unit_iff_Fred} give us an \(a\in \cA\) such that 
    \[
        \norm{x} \lesssim \norm{ax} + \norm{Dx}
    \]
    for all \(x\in \Dom D\). Let \(\chi\in\cA\) be such that \(a \chi = a\), then 
    \[
        \norm{(1-\chi)x} \lesssim \norm{a(1-\chi)x} + \norm{D(1-\chi)x} = \norm{D(1-\chi)x}.
    \]

    For the other direction, fix an approximate unit \(\set{\chi_n}\subseteq \cA\). Let \(A>0\) be such that
    \[
        \norm{(1 - \chi) x}\leq A \norm{D (1-\chi) x}
    \]
    and using \autoref{thm:compact_upper_bound_with_approx_unit} let \(M>0\) and \(N\) be such that 
    \[
        \norm{[D,\chi]x}\leq \frac{1}{2A}\norm{x} + \frac{1}{2A}\norm{Dx} + M\norm{\chi_n x}
    \]
    for all \(n>N\). Then
    \[
        \begin{split}
            \norm{x} 
            &\leq \norm{\chi x} + \norm{(1-\chi)x} \\
            &\leq \norm{\chi x} + A\norm{D(1-\chi)x} \\
            &\leq \norm{\chi x} + A\norm{(1-\chi)}\norm{Dx} + A\norm{[D,\chi]x} \\
            &\leq \norm{\chi x} + (A\norm{(1-\chi)} + \tfrac{1}{2})\norm{Dx} + \tfrac{1}{2}\norm{x} + M\norm{\chi_n x}
        \end{split}
    \]
    which shows that \(\chi^*\chi + \chi_n^*\chi_n + D^*D\) is invertible. Since \((\chi^*\chi + \chi_n^*\chi_n)^{\frac{1}{2}}\in A\) and \(\cA\) is dense in \(A\), \(a(1+D^*D)^{-\frac{1}{2}}\in \cK(\pE)\) for all \(a\in A\). Using \autoref{thm:characterization_of_Fredholm_abstract} we obtain that \(D\) is left-\(B\)-Fredholm.
\end{proof}

Similarly to the previous discussion regarding \autoref{thm:approx_unit_iff_Fred} found directly thereafter, \autoref{thm:upper_bound_when_restricted_by_approx_unit} is a direct abstraction of the first statement in \autoref{thm:characterization_of_Fredholm_geometric_extended} when considering \(C_c^\infty(M)\) acting on \(L^2(M;E)\) and \(L^2(M;F)\) for vector \(B\)-bundles \(E\) and \(F\).

\subsection{Persson's lemma}
\label{sec:perssonslemma}

The first statement in \autoref{thm:characterization_of_Fredholm_geometric_extended} is a statement about a norm estimate for sections with support outside a compact set. This relates to the classical result Person's lemma \cite{Persson_1960} in spectral theory, which gives a quantitative result on where the bottom of the essential spectrum is for a self-adjoint operator bounded from below.

In this section we will prove the second statement in \autoref{thm:characterization_of_Fredholm_geometric_extended} which is a perturbation result for left-\(B\)-Fredholm operators, and then use this result to show that Persson's lemma holds for far more operators than how the result usually is stated.

\begin{proof}[Proof of second statement in \autoref{thm:characterization_of_Fredholm_geometric_extended}]
    For a compact \(K\subseteq M\) we introduce the notation
    \[
        \Sigma_K \coloneq \inf_{\substack{f\in \Dom D \setminus \set{0} \\ \supp f \subseteq M\setminus K}} \frac{\norm{Df}}{\norm{f}}
    \]
    and note that if \(K\subseteq K'\) then \(\Sigma_K \leq \Sigma_{K'}\), hence \(\Sigma_K\leq \Sigma\) for any \(K\).

    The first statement in \autoref{thm:characterization_of_Fredholm_geometric_extended} (which is already proven, see \autoref{thm:upper_bound_when_restricted_by_approx_unit} and discussion thereafter) says that \(D\) is left-\(B\)-Fredholm if and only if there is a compact \(K_0\subseteq M\) such that \(\Sigma_{K_0}>0\). Hence, \(D\) is left-\(B\)-Fredholm if and only if \(\Sigma>0\).

    Now fix a \(T\in \cL(L^2(M;E))\) such that \(\norm{T}<\Sigma\). Then there is a compact \(K_1\subseteq M\) such that \(\norm{T}<\Sigma_{K_1}\). Take any \(f\in \Dom D\) with \(\supp f \subseteq M\setminus K_1\), then by definition \(\norm{Df}\geq \Sigma_{K_1}\norm{f}\). Hence,
    \[
        \norm{(D+T)f} \geq \norm{Df} - \norm{T}~\norm{f}\geq (\Sigma_{K_1} - \norm{T})\norm{f}
    \]
    and the first statement in \autoref{thm:characterization_of_Fredholm_geometric_extended} implies that \(D+T\) is left-\(B\)-Fredholm since \(\Sigma_{K_1} - \norm{T} > 0\).
\end{proof}

Note that \(T\) in the perturbation result of \autoref{thm:characterization_of_Fredholm_geometric_extended} could be taken to be an unbounded operator as well if one has control over the norm of \(T\) with respect to the graph-norm of \(D\).

\begin{corollary}[Persson's lemma]
    \label{thm:perssonslemma}
    Let \(D\) be a self-adjoint unbounded operator that is bounded from below on the Hilbert space \(L^2(M;E)\). Assume that \(\Dom D \subseteq H_{\loc}^s(M;E)\) for some \(s>0\) and for any \(a\in C_c^\infty(M)\) that \(a\Dom D\subseteq \Dom D\) and \([D,a](1+D^2)^{-\frac{1}{2}}\) is compact. Then \(\widetilde{\Sigma} = \inf \essspec(D)\) where
    \[
        \widetilde{\Sigma} \coloneq \sup_{\substack{K\subseteq M \\ \text{compact}}} \inf_{\substack{f\in \Dom D \setminus \set{0} \\ \supp f \subseteq M\setminus K}} \frac{\braket{f, Df}}{\norm{f}^2}
    \]
    and \(\essspec(D)\coloneq \setcond{\lambda\in\bR}{D-\lambda \text{ is not left-Fredholm}}\).
\end{corollary}

\begin{proof}
    First assume that \(D\) is positive, then
    \begin{equation}
        \label{eq:lowerbound_essspec}
        \widetilde{\Sigma} \leq \Sigma \leq \inf \essspec(D)
    \end{equation}
    using \autoref{thm:characterization_of_Fredholm_geometric_extended}. If \(D\) is instead bounded from below such that \(D+\gamma\) is positive, then \eqref{eq:lowerbound_essspec} also holds since adding \(\gamma\) to \(D\) simply shifts both \(\Sigma\) and \(\essspec(D)\) by \(\gamma\).

    To show the inequality in the other direction, take any \(\mu < \inf \essspec(D)\). Then \(P=\chi_{(-\infty, \widetilde{\Sigma}(D))}(D)\) is a finite rank projection since \(D\) is bounded from bellow and for any \(\epsilon>0\) we can choose a \(K_0\) such that \(\norm{Pf}<\epsilon\norm{f}\) for any \(f\) with \(\supp f \subseteq M\setminus K_0\). Hence,
    \[
        \begin{split}
            \braket{f, Df} 
            &= \braket{(1-P)f, D(1-P)f} + \braket{Pf, DPf} \geq \mu\norm{(1-P)f}^2 - \mu\norm{Pf}^2 \\
            &= \mu \norm{f}^2 - 2\mu\norm{Pf}^2 \geq \mu(1-2\epsilon)\norm{f}^2
        \end{split}
    \]
    if \(\supp f \subseteq M\setminus K_0\) and we obtain that
    \[
        \widetilde{\Sigma}(D) \geq \mu(1-2\epsilon).
    \]
    Since we can choose \(\epsilon\) arbitrarily small we see that \(\widetilde{\Sigma}(D) \geq \mu\), and since we can choose \(\mu\) arbitrarily close to \(\inf \essspec(D)\) we see that \(\widetilde{\Sigma}(D) \geq \inf \essspec(D)\).
\end{proof}

The part of the proof of \autoref{thm:perssonslemma} not depending on \autoref{thm:characterization_of_Fredholm_geometric_extended} is based on \cite[Apendix B.4]{Fournais_Helffer_2010}, and is the reason we restrict ourselves to Hilbert spaces in \autoref{thm:perssonslemma}.

\section{Unbounded Kasparov modules without locally compact resolvent}
\label{sec:KKtheory}

\subsection{Truly unbounded Kasparov modules}

The relative index theorem in \(KK\)-theory by Bunke \cite{Bunke_1995} was proven by directly constructing bounded Kasparov modules from unbounded operators instead of using unbounded Kasparov modules, as the operators used did not necessarily have compact resolvent, that is, locally compact resolvent when considering the unital algebra of the Higson compactification. To construct these bounded cycles, Bunke assumed that the operators were invertible at infinity. But with our characterization of Fredholmness, invertibility at infinity is implied by Fredholmness, and we arrive at the following altered definition of unbounded Kasparov modules. This definition was studied in \cite[Definition 2.4]{Wahl_2007} where we have taken the name from. 

\begin{definition}[Truly unbounded Kasparov module]
    \label{def:truly_unbounded_kasparov_modules}
    Let \(A\) and \(B\) be trivially graded \Cst-algebras. An even truly unbounded Kasparov \((A,B)\)-module is a triple \((\cA, \pE, D)\) where \(D\) is an odd self-adjoint regular operator on a graded Hilbert \((A,B)\)-bimodule \(\pE\) and \(\cA\subseteq A\) is a dense \(*\)-subalgebra such that the following are satisfied:
    \begin{enumerate}
        \item \(D\) is \(B\)-Fredholm;
        \item \(a \Dom D \subseteq \Dom D\) for all \(a\in \cA\);
        \item \([D,a](1+D^2)^{-\frac{1}{2}} \in \cK(\pE)\) for all \(a\in \cA\);
        \item for each \(a\in \cA\) there is an \(m\geq 1\) such that \([D,a](1+D^2)^{-\frac{1}{2}+\frac{1}{2m}} \in \cL(\pE)\).
    \end{enumerate}
    If there is an \(m\) that the last condition holds for all \(a\in \cA\), we say that the cycle is of order \(m\). An odd truly unbounded Kasparov \((A,B)\)-module is defined the same except that \(\pE\) is trivially graded and \(D\) is not required to be odd.
\end{definition}

The difference between truly unbounded Kasparov modules to the usual definition of unbounded Kasparov modules (or \(\epsilon\)-unbounded Kasparov modules \cite[Appendix A]{Goffeng_Mesland_2015}) is that the condition of locally compact resolvent, meaning that \(a(1+D^2)^{-\frac{1}{2}}\in \cK(\pE)\), is replaced by the two conditions that \(D\) is \(B\)-Fredholm and \([D,a](1+D^2)^{-\frac{1}{2}} \in \cK(\pE)\). In particular, if \(D\) has compact resolvent and fulfills one definition, it also fulfills the other. The benefit of truly unbounded Kasparov modules is that we easily can pass to the unitalization of \(A\) since all expressions with interaction between \(A\) and \(D\) involve commutators or mapping the domain to the domain. As we will see, truly unbounded Kasparov modules define a class in \(KK\)-theory and then one can use the functorial map \(KK^0(A^+,B)\to K_0(B)\) which is taking the \(B\)-Fredholm index of the operator \(D\).

\begin{remark}
Note however that adding the unit cannot be seen as a map from \(KK^*(A,B)\) to \(KK^*(A^+,B)\). For a concrete example to see that this cannot be, consider an elliptic differential operator \(D\) over a compact manifold with boundary \(M\). Then any closed realization \(D_e\) of \(D\) defines the same class \([D_e]\in K^*(C_0(M))\) \cite[Section 3.5]{Fries_2025}. In particular, any Fredholm realization corresponding to a boundary condition on \(D\) give the same class in \(K^*(C_0(M))\) and can at the same time be lifted as a cycle to define a class in \(K^*(C_0(M)^+)\). However, the classes in \(K^*(C_0(M)^+)\) cannot be the same as different boundary conditions on the same differential operator can result in different Fredholm index. For example different spectral cuts for an Atiyah-Patodi-Singer boundary condition on a Dirac operator produce different indices as quantified by spectral flow.

That any truly unbounded Kasparov module for \(A\) also defines a cycle for \(A^+\) should instead be seen as finding a preimage in the map \(KK^*(A^+,B) \to KK^*(A,B)\).
\end{remark}

The definitions of \(\epsilon\)-unbounded Kasparov modules and truly unbounded Kasparov modules are not the same in the sense that there are \(\epsilon\)-unbounded Kasparov modules that are not truly unbounded Kasparov modules and vice versa. For example, on one hand there are Dirac operators having full spectrum with locally compact resolvent with respect to \(C_c^\infty(M)\), and on the other hand there are Fredholm operators that do not have compact resolvent, that is, locally compact resolvent with respect to a unitalization. 

\begin{example}
    Sometimes it is possible to pass a truly unbounded Kasparov module to a larger unitalization then \(A^+\). For instance, let \(\slashed{D}\) be the Dirac operator constructed from the spin structure of a complete spin manifold \(M\). Then
    \[
        (C_c^\infty(M),L^2(M;\slashed{S}),\slashed{D})
    \]
    defines an unbounded Kasparov \((C_0(M),\bC)\)-module. In \cite[Example 2.15]{Fries_2025}, it is also shown that \([\slashed{D},a]\colon \Dom \slashed{D} \to L^2(M;\slashed{S})\) is compact for any \(a\in C_h^\infty(M)\coloneq \setcond{b\in C_b^\infty(M)}{db\in C_0^\infty(M;T^*M)}\). Hence, if \(M\) has positive scalar curvature outside a compact, then \(\slashed{D}\) is Fredholm and 
    \[
        (C_h^\infty(M),L^2(M;\slashed{S}),\slashed{D})
    \]
    defines a truly unbounded Kasparov \((C_h(M),\bC)\)-module where \(C_h(M)=\cl{C_h^\infty(M)}\) is called the Higson compactification \cite[Chapter 5]{Roe_1993}.

    In particular, with in \cite{Fries_2025} in mind we see that \(\d [\slashed{D}] \in K^{*-1}(C(\nu M))\) is an obstruction of existence of scalar curvature strictly positive outside a compact, where \(\nu M\) is the Higson corona defined from \(C(\nu M)=C_h(M)/C_0(M)\) and \(\d\colon K^*(C_0(M))\to K^{*-1}(C(\nu M))\) is the boundary map in \(K\)-homology.

\end{example}

\begin{remark}
The naming of “truly unbounded” is misleading, as the operator \(D\) could as well be bounded. In fact, given any bounded Kasparov module \((\pE, F)\), meaning an operator \(F\in \cL(\pE)\) such that \([F,a],a(F^* - F),a(F^2 - 1)\in \cK(\pE)\) for all \(a\in A\), we can modify \(D\) by \(A\)-compact perturbations to ensure that \(F=F^*\) and \(F^2-1 = 0\). See for example \cite[Section 8.3]{Higson_Roe_2000}. Any such cycles \((\pE,F)\) also define a truly unbounded Kasparov module. As a consequence, any element in \(KK^*(A,B)\) can be represented by a truly unbounded Kasparov module. One can compare the effort to show this to the work in \cite{Baaj_Julg_1983} showing that the usual notion unbounded \(KK\)-cycles can represent any class in \(KK^*(A,B)\).
\end{remark}

Note also that as a model for \(K\)-theory, a truly unbounded Kasparov \((\bC,B)\)-module simply consists of a \(B\)-Fredholm operator. Looking at real \Cst-algebras, or Real \Cst-algebras, one could also use truly unbounded Kasparov modules as a model for \(KKO\)-theory or \(KKR\)-theory.

\subsection{Alternatives to the bounded transform}
\label{sec:bounded_cycles}

It is important to note that the bounded transform \(F_{D}\coloneq D(1+D^2)^{-\frac{1}{2}}\) does not generally give a bounded \(KK\)-cycle from a truly unbounded Kasparov module. If the bounded transform \(F_D\) of a self-adjoint regular operator \(D\) would have given a bounded cycle, then \(D\) would automatically have locally compact resolvent since \(a(F_D^2 - 1) = -a(1+D^2)^{-1}\), which does not need to be the case for truly unbounded Kasparov modules.

To obtain bounded cycles from a truly unbounded Kasparov module, we will present two recipes. We will show that any choices within the recipes or between these recipes result in the same class in \(KK\)-theory.

The first way to construct a bounded cycle uses functional calculus and is originally from \cite[Definition 2.4]{Wahl_2007}.

\begin{proposition}
    \label{thm:Wahl_gives_bounded_cycles}
    For a truly unbounded Kasparov \((A,B)\)-module \((\cA, \pE, D)\) there exists an odd real-valued non-decreasing function \(\chi\in C(\bR)\) with well-defined limits at \(\pm \infty\) such that \(\chi(D)^2 - 1\in \cK(\pE)\). Then \((\pE, \chi(D))\) defines a bounded Kasparov \((A,B)\)-module and the class \([(\pE, \chi(D))]\in KK^*(A,B)\) is independent of choice of such \(\chi\).
\end{proposition}

To prove \autoref{thm:Wahl_gives_bounded_cycles}, Wahl showed that \([\chi(D),a]\in \cK(\pE)\) for any function \(\chi\in C(\bR)\) with well-defined limits at \(\pm \infty\) \cite[the proof after Definition 2.4]{Wahl_2007}. The same method will be used later in the proof of \autoref{thm:diff_of_func_is_comp}. The independence of \(\chi\) was not presented in \cite{Wahl_2007} and will be proven in the following two lemmas.

Note that requiring \(\chi\) to be real-valued and non-decreasing in \autoref{thm:Wahl_gives_bounded_cycles} can just as well be relaxed to \(\chi^{-1}(-1)\subseteq (-\infty,0)\) and \(\chi^{-1}(1)\subseteq (0,\infty)\).

\begin{lemma}
    \label{thm:perturb_funccalc_compact}
    Let \(D\) be a self-adjoint operator on a Hilbert module \(\pE\) and \(f\in C_b(\bR)\) with well-defined limits at \(\pm \infty\) be such that \(f(D)\in \cK(\pE)\), then \(g(D)\in \cK(\pE)\) for any \(g\in C_b(\bR)\) with well-defined limits at \(\pm \infty\) and \(\supp g \subseteq \supp f\). 
\end{lemma}

\begin{proof}
    To show this, we will find a sequence \(g_n\in C_b(\bR)\) such that \(g_n\to g\) and \(g_n = b_n f\) for some \(b_n\in C_b(\bR)\). Then \(g_n(D)\in \cK(\pE)\) converges to \(g(D)\) in norm showing that \(g(D)\in \cK(\pE)\).

    To construct these functions it is easier to transfer the problem to \(C([0,1])\) instead as this is equivalent. Consider \(f\in C([0,1])\) and \(g\in C_0(\interior{(\supp f)})\). Let \(g_n\in C_c(\interior{(\supp f)})\) be a sequence such that \(g_n\to g\). Define functions \(b_n\) as \(b_n(x) = \frac{g_n(x)}{f(x)}\) if \(x\in \supp g_n\) and \(b_n(x)=0\) otherwise. Then \(b_n\) is continuous and \(\abs{b_n}\leq \norm{g_n}\sup_{x\in \supp g_n} \abs{f(x)}^{-1}\) which is bounded since \(\supp g_n\) is compact. In particular, \(g_n = b_n f\) and \(b_n\in C([0,1])\), and we are done.
\end{proof}

\begin{lemma}
    \label{thm:chi_exists}
    Let \(D\) be a \(B\)-Fredholm self-adjoint operator on a Hilbert \(B\)-module \(\pE\). Then there exists an odd real-valued non-decreasing function \(\chi\in C_b(\bR)\) with well-defined limits at \(\pm \infty\) such that \(\chi(D)^2 - 1\in \cK(\pE)\).

    Furthermore, for any two such functions \(\chi\) and \(\chi'\), \(\chi(D)-\chi'(D)\in \cK(\pE)\).
\end{lemma}

\begin{proof}
    Since \(D\) is \(B\)-Fredholm, \(D(1+D^2)^{-\frac{1}{2}}\) is invertible in the Calkin algebra \(\cL(\pE)/\cK(\pE)\). Hence, there is a neighborhood \(U\) of zero such that \(\phi(D)\in \cK(\pE)\) for any \(\phi\in C_c(\bR)\) with \(\supp \phi\subseteq U\). For such a \(\phi\in C_c(\bR)\) that is even, non-decreasing away from zero and \(\phi(0)>0\), \(\chi(x) =  x(\phi(x)+x^2)^{-\frac{1}{2}}\) satisfies the requirement since \(\chi(x) = \sgn x\) outside \(U\).

    Now take any \(\chi\in C_b(\bR)\) satisfying the requirements. Note that \(\chi(x)=\sgn x\) outside \(\supp (\chi^2 - 1)\). Let \(U\) be a neighborhood of zero such that \(\chi^2 \neq 1\) on \(U\) and there is a \(\phi\in C_c(\bR)\) that is even, non-decreasing away from zero, \(\phi(0)>0\) and \(\supp \phi \subseteq U\). Consider the function \(g(x)= x(\phi(x)+x^2)^{-\frac{1}{2}} - \chi(x)\). Then \(\supp g \subseteq \supp (\chi^2 - 1) \cup U \subseteq \supp (\chi^2 - 1)\), and hence \(g(D)\in \cK(\pE)\) by \autoref{thm:perturb_funccalc_compact}. Since the same \(\phi\) can be chosen for any two such \(\chi\), this is enough.
\end{proof}

The other way to obtain a bounded cycle from a truly unbounded Kasparov module is following the construction of the bounded cycle in \cite{Bunke_1995} using invertibility at infinity, which now with our the characterization of Fredholmness in \autoref{thm:characterization_of_Fredholm_abstract} works for any truly unbounded Kasparov module. 

\begin{proposition}
    \label{thm:gives_bounded_cycles}
    For a truly unbounded Kasparov \((A,B)\)-module \((\cA, \pE, D)\) there exists a self-adjoint \(k\in \cL(\pE)\cap\cK(\Dom D, \pE)\) such that \(k+D^2\) is invertible. Define
    \[
        F_{D,k} \coloneq D(k + D^2)^{-\frac{1}{2}},
    \]
    then \((\pE, F_{D,k})\) defines a bounded Kasparov \((A,B)\)-module and the class \([(\pE,F_{D,k})]\in KK^*(A,B)\) is independent of choice of such \(k\).
\end{proposition}

For differential operators, the construction of \(F_{D,k}\) is very geometric since \(k\) in \autoref{thm:gives_bounded_cycles} can be chosen to be an element of \(C_c^\infty(M)\) using \autoref{thm:characterization_of_Fredholm_geometric}. Note however that \(k\) does not have to have any relation to the algebra \(A\) or subalgebra \(\cA\). In particular, some cycles in \autoref{thm:Wahl_gives_bounded_cycles} can be constructed using a \(k\) as in \autoref{thm:gives_bounded_cycles}. Specifically, we can let \(k=\phi(D)\) for some \(\phi\) as in the proof of \autoref{thm:chi_exists} and let \(\chi(x)=x(\phi(x) + x^2)^{-\frac{1}{2}}\). This shows that all bounded cycles in this section give the same class in \(KK\)-theory. In fact, they are compact perturbations of each other. 

Note that if \(T\in \cL(\pE)\) is a positive element such that \(T+D^*D\) is invertible, then \(\ind D = \ind D(T+D^*D)^{-\frac{1}{2}}\) which shows that all bounded cycles in this section have the same index as \(D\).

In the case when \(B = \bC\) and we are considering operators on Hilbert spaces, then the sign function (defined as \(\sgn(0)=0\)) is continuous on the spectrum of a Fredholm operator. In particular, \(\sgn(D)=D(P_{\ker D} + D^2)^{-\frac{1}{2}}\) defines a bounded cycle as seen by using \(k=P_{\ker D}\) in \autoref{thm:gives_bounded_cycles}. Such bounded cycles have appeared in previous literature, for example \cite[Proposition 3.1.2]{Lafforgue_2002}.

The proof of \autoref{thm:gives_bounded_cycles} is essentially the same as in \cite{Bunke_1995}, but we will reprove it here for the sake of completeness and to clarify that we do not use anything extra. Firstly, we will show that for self-adjoint \(k,k'\in \cL(\pE)\cap \cK(\Dom D, \pE)\) such that \(k+D^2\) and \(k'+D^2\) are invertible that 
\[
    F_{D,k} - F_{D,k'} \in \cK(\pE).
\]
This lets us restrict attention to \(k\in \CompInv(D)\) defined in \eqref{eq:compinv_def} which additionally requires that \(k \Dom D\subseteq \Dom D\) and show that
\[
    F_{D,k}^2 - 1, F_{D,k}^* - F_{D,k}, [F_{D,k}, a]
\]
are compact operators for \(k\in \CompInv(D)\). 

As in the usual fashion in these type of proofs we will utilize the resolvent integral
\[
    T^{-\frac{1}{2}} = \frac{2}{\pi} \int_0^\infty(\lambda^2 + T)^{-1} \, \rd\lambda
\]
if \(T\) is an invertible positive unbounded regular operator, a method originating from \cite{Baaj_Julg_1983}. For \(x\in \pE\) and \(y\in \Dom D\) we see that
\[
    \begin{split}
    \braket{F_{D,k}x, y} 
    &= \braket{x, (k + D^2)^{-\frac{1}{2}}D y} \\
    &= \braket{x, \frac{2}{\pi} \int_0^\infty (\lambda^2 + k + D^2)^{-1} Dy \, \rd\lambda} \\
    &= \braket{\frac{2}{\pi} \int_0^\infty D(\lambda^2 + k + D^2)^{-1}x \, \rd\lambda,  y } \\
    \end{split}
\]
from which we obtain that
\[
    \label{eq:pointwise_integral_formula}
    F_{D,k} = \frac{2}{\pi} \int_0^\infty D (\lambda^2 + k + D^2)^{-1} \, \rd\lambda
\]
holds pointwise on \(\pE\). To show that an expression with \(F_{D,k}\) is compact we will look at the pointwise integral formula in \eqref{eq:pointwise_integral_formula} on a dense subspace and massage it until we can conclude that it is a norm convergent integral over compact operators. This massaging will in large be done using the following lemma.

\begin{lemma}
    \label{thm:commutator_with_ish_resolvent}
    Let \(D\) be a self-adjoint regular operator on \(\pE\) and \(T\in \cL(\pE)\) a self-adjoint element such that \(T + D^2\) is invertible and \(T\Dom D \subseteq \Dom D\), then 
    \[
        (DR_T)^* = DR_T + R_T[D,T]R_T
    \]
    where \(R_T = (T + D^2)^{-1}\). In particular, \([D,R_T] = -R_T[D,T]R_T\) on \(\Dom D\).
\end{lemma}

\begin{proof}
    For \(u, v\in \pE\) we have that
    \[
        \begin{split}
            \braket{u, DR_T v} 
            &= \braket{(T + D^2)R_Tu, DR_T v} \\
            &= \braket{DR_Tu, D^2R_T v} + \braket{TR_Tu, DR_T v} \\
            &= \braket{DR_Tu, v} + \braket{DR_Tu, -TR_T v} + \braket{TR_Tu, DR_T v} \\
            &= \braket{(DR_T + R_T[D,T]R_T)u, v}, \\
        \end{split}
    \]
    and we are done.
\end{proof}

\begin{lemma}
    \label{thm:different_k_compact_pretubation}
    For a self-adjoint regular operator \(D\) on \(\pE\) and self-adjoint \(k, k'\in \cL(\pE)\cap \cK(\Dom D, \pE)\) such that \(k+D^2\) and \(k'^2 + D^2\) are invertible
    \[
        F_{D, k} - F_{D, k'} \in \cK(\pE).
    \]
\end{lemma}

\begin{proof}
    Pointwise on \(\pE\) we have that
    \[
        \begin{split}
            F_{D, k} - F_{D, k'}
            &= \frac{2}{\pi} \int_0^\infty D((\lambda^2 + k + D^2)^{-1} - (\lambda^2 + k' + D^2)^{-1}) \, \rd\lambda \\
            &= \frac{2}{\pi} \int_0^\infty D(\lambda^2 + k + D^2)^{-1}(k - k')(\lambda^2 + k' + D^2)^{-1} \, \rd\lambda \\
        \end{split}
    \]
    where the integrand is compact since \(k (1 + D^2)^{-1}\) and \(k' (1 + D^2)^{-1}\) are compact.
\end{proof}

An important note is that in the case when \(D\) has locally compact resolvent, that is \(a(1+D^2)^{-\frac{1}{2}}\in \cK(\pE)\), then any \(F_{D, k}\) give rise to the same \(KK\)-class as the usual bounded transform \(F_{D,1}\) which can be shown similarly as \autoref{thm:different_k_compact_pretubation}.

\begin{lemma}
    \label{thm:selfadjoint_up_to_compact}
    For a self-adjoint regular operator \(D\) on \(\pE\) and a \(k\in \CompInv(D)\)
\[
    F_{D, k} - F_{D, k}^* \in \cK(\pE).
\]
\end{lemma}

\begin{proof}
    Pointwise on \(\Dom D\) we have that
    \[
        \begin{split}
            F_{D, k} - F_{D, k}^*
            &= \frac{2}{\pi} \int_0^\infty [D, (\lambda^2 + k + D^2)^{-1}] \, \rd\lambda \\
            &= - \frac{2}{\pi} \int_0^\infty (\lambda^2 + k + D^2)^{-1}[D,k] (\lambda^2 + k + D^2)^{-1} \, \rd\lambda
        \end{split}
    \]
    using \autoref{thm:commutator_with_ish_resolvent} where the integrand is compact since \(k (1 + D^2)^{-1}\) is compact and the integral is norm convergent.
\end{proof}

\begin{lemma}
    For a self-adjoint regular operator \(D\)  on \(\pE\) and a \(k\in \CompInv(D)\)
    \[
        F_{D, k}^2 - 1 \in \cK(\pE).
    \]
\end{lemma}
    
\begin{proof}
    Using \autoref{thm:selfadjoint_up_to_compact} we can instead look at \(F_{D, k}^*F_{D, k} - 1\). Pointwise on \(\Dom D^2\) we have that
    \[
        \begin{split}
            F_{D, k}^*F_{D, k} 
            &= F_{D, k}^*F_{D, k} (k + D^2)^{-1}(k + D^2) \\
            &= F_{D, k}^* D(k + D^2)^{-\frac{1}{2}} (k + D^2)^{-1}(k + D^2) \\
            &= (k + D^2)^{-\frac{1}{2}}D^2(k + D^2)^{-1} (k + D^2)^{-\frac{1}{2}}(k + D^2) \\
            &= 1 - (k + D^2)^{-\frac{1}{2}}k (k + D^2)^{-\frac{1}{2}} \\
        \end{split}
    \]
    where the second term is compact since \(k (1 + D^2)^{-\frac{1}{2}}\) is compact.
\end{proof}

\begin{lemma}
    \label{thm:commutator_compact}
    For a truly unbounded Kasparov module \((\cA,\pE,D)\) and \(k\in \CompInv(D)\) for any \(a\in \cA\)
    \[
        [F_{D, k}, a] \in \cK(\pE).
    \]
\end{lemma}
    
\begin{proof}
    Let \(R(\lambda) = (\lambda^2 + k + D^2)^{-1}\), then pointwise on \(\pE\) we have that
    \[
        [F_{D, k},a] = \frac{2}{\pi} \int_0^\infty [DR(\lambda), a] \, \rd\lambda.
    \]
    We have that
    \[
            [k, R(\lambda)] 
            = -[D^2, R(\lambda)]
            = D R(\lambda) [D, k] R(\lambda) + R(\lambda) [D, k] R(\lambda)D
    \]
    holds pointwise on \(\Dom D^2\) and extends by continuity to \(\Dom D\). 
    Using this we can rewrite \(a R(\lambda)\) and \(R(\lambda)a\) pointwise on \(\pE\) as
    \[
        \begin{split}
            a R(\lambda)
            &= (\lambda^2 + k + D^2) R(\lambda) a R(\lambda) \\
            &= \lambda^2 R(\lambda) a R(\lambda) \\
            &\hphantom{=}+ R(\lambda) k a R(\lambda) \\
            &\hphantom{=}+ R(\lambda) [D, k] R(\lambda) [D, a] R(\lambda) \\
            &\hphantom{=}+ D R(\lambda) [D, a] R(\lambda) \\
            &\hphantom{=}+ R(\lambda) D a D R(\lambda)
        \end{split}
        \quad
        \begin{split}
            R(\lambda) a
            &= R(\lambda) a (\lambda^2 + k + D^2) R(\lambda) \\
            &= \lambda^2 R(\lambda) a R(\lambda) \\
            &\hphantom{=}+ R(\lambda) a k R(\lambda) \\
            &\hphantom{=}- R(\lambda) [D,a] D R(\lambda) \\
            &\hphantom{=}+ R(\lambda) D a D R(\lambda) \\
            &
        \end{split}
    \]
    which we combine as
    \[
        \begin{split}
            [D R(\lambda), a]
            &= [D, a]R(\lambda) + D[R(\lambda), a] \\
            &= D R(\lambda) a k R(\lambda) \\
            &\hphantom{=}- D R(\lambda) k a R(\lambda) \\
            &\hphantom{=}- D R(\lambda) [D,a] D R(\lambda) \\
            &\hphantom{=}+ (\lambda^2 + k) R(\lambda) [D, a] R(\lambda) \\
            &\hphantom{=}- D R(\lambda)[D,k]R(\lambda) [D,a] R(\lambda) \\
            &= D R(\lambda) a k R(\lambda)^{\frac{1}{2} - \frac{\epsilon}{2}} R(\lambda)^{\frac{1}{2} + \frac{\epsilon}{2}} \\
            &\hphantom{=}- D R(\lambda) k a R(\lambda)^{\frac{1}{2} - \frac{\epsilon}{2}} R(\lambda)^{\frac{1}{2} + \frac{\epsilon}{2}} \\
            &\hphantom{=}- (D R(\lambda)^\frac{1}{2}) R(\lambda)^\frac{\epsilon}{2} (R(\lambda)^{\frac{1}{2}-\frac{\epsilon}{2}} [D,a]) (D R(\lambda)^\frac{1}{2}) R(\lambda)^\frac{1}{2} \\
            &\hphantom{=}+ (\lambda^2 + k) R(\lambda) ([D,a] R(\lambda)^{\frac{1}{2}-\frac{\epsilon}{2}}) R(\lambda)^{\frac{1}{2}+\frac{\epsilon}{2}} \\
            &\hphantom{=}- D R(\lambda) [D,k] R(\lambda) ([D,a] R(\lambda)^{\frac{1}{2}-\frac{\epsilon}{2}}) R(\lambda)^{\frac{1}{2}+\frac{\epsilon}{2}}. \\
        \end{split}
    \]
    Using that \(\norm{R(\lambda)^\alpha}\leq (\lambda^2 + \norm{k})^{-\alpha}\) for \(\alpha>0\) we obtain the norm estimate
    \[
        \norm{[D R(\lambda), a]} \lesssim (\lambda^2 + \norm{k})^{-\frac{1}{2}-\frac{\epsilon}{2}}
    \]
    and \([D R(\lambda), a]\) is compact since \(k R(\lambda)^{\frac{1}{2}}\) and \([D,a] R(\lambda)^{\frac{1}{2}}\) are compact.
\end{proof}

Note also that a big complication in the proof of \autoref{thm:commutator_compact} is that we do not assume that \(a\) preserve \(\Dom D^2\). This concludes the proof of \autoref{thm:gives_bounded_cycles}.

\subsection{The relative index theorem in \texorpdfstring{\(KK\)}{KK}-theory}

Examining the requirements for \autoref{thm:relative_index_for_differential_operators}, we note that the natural algebra for these cycles to be associated with is the continuous functions on a two-point compactification of \(U\). Let us explain this further.

For fixed \(i,j\), considering \(U, V_i, W_j\) in the diagram \eqref{eq:the_diagram} as subsets of \(M_{ij}\) and define the two-point compactification of \(U\) as \(U^{2\pt} = U / ((M_{ij}\setminus V_i) \cup (M_{ij}\setminus W_j))\) which is homeomorhpic to \(U\cup\set{\pt_1, \pt_2}\) with topology induced by the map
\begin{equation}
    \label{eq:twopoint_comp_map_inducing_topology}
    x\mapsto \begin{cases}
        x &, x\in U \\
        \pt_1 &, x\in M_{ij}\setminus V_i \\
        \pt_2 &, x\in M_{ij}\setminus W_j
    \end{cases}
\end{equation}
This construction is independent of \(i\) and \(j\) since the boundary \(\d U\) consists of the disjoint and topologically separated sets \(\d U \cap V_i = \d W_j\) and \(\d U \cap W_j=\d V_i\). Note that \(U^{2\pt}\) is compact since \(U\) is assumed to be precompact. Now the algebra \(C(U^{2\pt})\) can be seen as a subalgebra of functions \(C_b(M_{ij})\) using the pullback of \eqref{eq:twopoint_comp_map_inducing_topology} and \(C_0(U)\) is an ideal in \(C(U^{2\pi})\). This has similarities to the end compactification \(X^\text{end} = \dirlim_{K~\text{compact}} X \cup \pi_0(X\setminus K)\).

The requirements for \autoref{thm:relative_index_for_differential_operators} are essentially that \(D_{11}\) and \(D_{22}\) are local operators that define truly unbounded Kasparov \((C(U^{2\pt}),B)\)-modules that agree on \(U\) and \(D_{12}, D_{21}\) are constructed by gluing. Note that the unitary in \autoref{thm:bunkes_unitary} used to prove \autoref{thm:relative_index_for_differential_operators} is a unitary in \(M_2(C(U^{2\pt}))\) that maps a core of \(D_{11}\oplus D_{22}\) onto a core of \(D_{12}\oplus D_{21}\). One could write out abstract versions of the requirements needed for \autoref{thm:relative_index_for_differential_operators}, however, this does not seem to be very illuminating. See \cite{Nazaikinskii_2015} for a neat abstract setup of the relative index theorem for bounded cycles, which given the work in this paper we expect should be possible to recreate in the unbounded setting. Here we will instead show the key statement regarding relatively compact perturbations on the level of truly unbounded Kasparov modules.

\begin{lemma}
    \label{thm:diff_of_func_is_comp}
    Let \(D\) and \(D'\) be self-adjoint regular operator on a Hilbert \(B\)-modules \(\pE\) and \(\pE'\), respectively. Assume that \(U\in \cL(\pE, \pE')\) is a unitary such that 
    \begin{enumerate}
        \item \(U\) maps a core \(\cE\) of \(D\) onto a core of \(D'\);
        \item \((U^*D'U-D)(1+D^2)^{-\frac{1}{2}}\) extends to a compact operator in \(\cK(\pE)\);
        \item \((U^*D'U-D)(1+D^2)^{-\frac{1}{2}+\epsilon}\) extends to a bounded operator in \(\cL(\pE)\) for some \(\epsilon>0\).
    \end{enumerate}
    Then 
    \begin{equation}
        \label{eq:chi_diff_compact}
        U^*\chi(D')U - \chi(D)\in \cK(\pE)
    \end{equation}
    for any \(\chi\in C_b(\bR)\) with well-defined limits \(\lim_{x\to\pm \infty} \chi(x)\).
\end{lemma}

\begin{proof}
By \autoref{thm:relativly_compact_pert_domains}, \(U\) restricts to an operator in \(\cL(\Dom D, \Dom D')\) with inverse \(U^*\). Since \(U^*\chi(D')U = \chi(U^*D'U)\) we can reduce the statement to the case when \(U=1\). Note first that if \eqref{eq:chi_diff_compact} hold for \(\chi_1\) and \(\chi_2\), then it holds for \(\chi_1 + \chi_2\) and \(\chi_1\chi_2\) where the later can be seen since
\[
    \begin{split}
    \chi_1(D') \chi_2(D') - \chi_1(D) \chi_2(D) 
    &= (\chi_1(D') - \chi_1(D)) \chi_2(D') \\
    &\hphantom{=}+ \chi_1(D) (\chi_2(D') - \chi_2(D)) 
    \end{split}
\]
and \(\cK(\pE)\) is an ideal in \(\cL(\pE)\). Now since the algebra generated by \((x+i)^{-1}\) and \((x-i)^{-1}\) is dense in \(C_0(\bR)\) and
\[
    (D'\pm i)^{-1} - (D \pm i)^{-1} = (D'\pm i)^{-1} (D - D') (D \pm i)^{-1}
\]
is compact, we can with continuity conclude that \eqref{eq:chi_diff_compact} holds for any \(\chi\in C_0(\bR)\). Equation \eqref{eq:chi_diff_compact} also holds trivially for constant \(\chi\) and if we show that \eqref{eq:chi_diff_compact} holds for \(\chi(x)=x(1+x^2)^{-\frac{1}{2}}\) then we are done. That
\[
    \begin{split}
    D'(1+D'^2)^{-\frac{1}{2}} - D(1+D^2)^{-\frac{1}{2}}
    &= 
    (D' - D)(1+D'^2)^{-\frac{1}{2}} \\
    &\hphantom{=}+ D((1+D'^2)^{-\frac{1}{2}} - (1+D^2)^{-\frac{1}{2}})
    \end{split}
\]
is compact can be concluded since the pointwise expression
\[
    \begin{split}
        &D((1+D'^2)^{-\frac{1}{2}} - (1+D^2)^{-\frac{1}{2}}) \\
        &= \frac{2}{\pi} \int_0^\infty D((1+\lambda^2+D'^2)^{-1} - (1+\lambda^2+D^2)^{-1}) \, \rd\lambda \\
        &= \frac{1}{\pi} \int_1^\infty D((t^2+D'^2)^{-1} - (t^2+D^2)^{-1}) \frac{t \rd t}{\sqrt{t^2-1}} \\
        &= \frac{1}{\pi} \int_1^\infty D \bigg(
            (D' + it)^{-1}(D' - it)^{-1} - (D + it)^{-1}(D - it)^{-1}
            \bigg) \frac{t \rd t}{\sqrt{t^2-1}} \\
        &= \frac{1}{\pi} \int_1^\infty D \bigg(
            (D' + it)^{-1}(D-D')(D+it)^{-1}(D' - it)^{-1}
            \\&\phantom{\frac{1}{\pi} \int_1^\infty} +(D + it)^{-1}(D' - it)^{-1}(D-D')(D - it)^{-1}
            \bigg) \frac{t \rd t}{\sqrt{t^2-1}} \\
    \end{split}
\]
is a norm-convergent integral of compact operators.
\end{proof}

\begin{proposition}
    \label{thm:relatively_comp_pert_KK}
    Let \(D, D'\) and \(U\) be as in \autoref{thm:diff_of_func_is_comp}. If \((\cA, \pE, D)\) defines a truly unbounded Kasparov \((A,B)\)-module then so does \((\cA', \pE', D')\) for \(\cA' \coloneq U\cA U^*\) and they give the same class in \(KK^*(A,B)\).
\end{proposition}

\begin{proof}
    By \autoref{thm:relativly_compact_pert_domains}, \(U\) restricts to an invertible element in \(\cL(\Dom D, \Dom D')\) with inverse \(U^*\) and hence \(a'\Dom D'\subseteq \Dom D'\) for any \(a'\in \cA'\). Using complex interpolation we also see that \(U\) is bounded as an invertible element in \(\cL(\Dom \abs{D}^t, \Dom \abs{D'}^t)\) for \(0\leq t \leq 1\). Now for \(a\in \cA\) we see that on \(U\cE\)
    \[
        [D', UaU^*] = U[U^*D'U - D, a]U^* + U[D, a]U^*
    \]
    from which we obtain our required commutator conditions. Also, \(D'\) is \(B\)-Fredholm by \autoref{thm:relativly_compact_pert_Fredholm}.

    That the cycles define the same class can be seen from \autoref{thm:diff_of_func_is_comp}.
\end{proof}

Note that in addition to \(D'\) being a relatively compact perturbation of \(D\) in \autoref{thm:diff_of_func_is_comp} we also require that their difference is bounded on \(\Dom \abs{D}^t\) for some \(t<1\). This condition is only needed to show that the classes are the same in \autoref{thm:relatively_comp_pert_KK}, not to obtain a truly unbounded Kasparov module, and is expected by the author to be an artifact of the method used. In the case when \(A=\bC\), we can see that \(D\) and \(D'\) define the same class since they have the same index and therefore their bounded transforms are homotopic by non-commutative Atiyah-Jänich theorem \cite[Section 4.4]{Gracia-Bondia_Varilly_Figueroa_2001}.

\providecommand{\bysame}{\leavevmode\hbox to3em{\hrulefill}\thinspace}

\end{document}